\RequirePackage{fixltx2e}
\documentclass[10pt]{article}
\usepackage[latin9]{inputenc}
\usepackage{float}
\usepackage{amsmath}
\usepackage{amsthm}
\usepackage{amssymb}
\usepackage[a4paper]{geometry}
\makeatletter
\theoremstyle{plain}
\newtheorem{thm}{\protect\theoremname}[section]
\theoremstyle{remark}
\newtheorem{rem}[thm]{\protect\remarkname}
\theoremstyle{definition}
\newtheorem{defn}[thm]{\protect\definitionname}
\theoremstyle{plain}
\newtheorem{cor}[thm]{\protect\corollaryname}
\theoremstyle{plain}
\newtheorem{lem}[thm]{\protect\lemmaname}
\theoremstyle{plain}
\newtheorem{prop}[thm]{\protect\propositionname}
\theoremstyle{definition}
\newtheorem{example}[thm]{\protect\examplename}

\@ifundefined{date}{}{\date{}}
\usepackage{graphicx}\usepackage[colorlinks=true,linkcolor=blue,citecolor=red]{hyperref}\usepackage{cleveref}
\usepackage{fancyhdr}\usepackage{setspace}\usepackage{mathrsfs}\usepackage{enumerate}\usepackage{tikz}\usepackage{fixltx2e}\usepackage{amsthm}
\usepackage{bbm}

\usetikzlibrary{decorations.markings}\usetikzlibrary{shapes.misc}
\makeatother

\providecommand{\corollaryname}{Corollary}
\providecommand{\definitionname}{Definition}
\providecommand{\examplename}{Example}
\providecommand{\lemmaname}{Lemma}
\providecommand{\propositionname}{Proposition}
\providecommand{\remarkname}{Remark}
\providecommand{\theoremname}{Theorem}

\begin{document}
\global\long\def\C{\mathbb{C}}%
\global\long\def\cnk{\left(\C^{n}\right)^{\otimes k}}%
\global\long\def\P{P_{k}(n)}%
\global\long\def\A{A_{k}(n)}%
\global\long\def\element{\left(e_{1}-e_{2}\right)\otimes\dots\otimes\left(e_{2k-1}-e_{2k}\right)}%
\global\long\def\Sn{S_{n}}%
\global\long\def\endsn{\mathrm{End}{}_{\Sn}}%
\global\long\def\xilambda{\sum_{\sigma\in S_{k}}\frac{d_{\lambda}}{k!}\chi^{\lambda}(\sigma)\left[v_{\sigma(1)}\otimes\dots\otimes v_{\sigma(k)}\right]}%
\global\long\def\p{\mathscr{P}}%
\global\long\def\dkn{\left\langle e_{i_{1}}\otimes\dots\otimes e_{i_{k}}:|\{i_{1},\dots,i_{k}\}|=k\right\rangle }%
\global\long\def\normxilambda{\xi_{\lambda}^{\mathrm{norm}}}%
\global\long\def\skdelta{S_{k}^{\Delta}}%
\global\long\def\vlambdacheck{\check{V}^{\lambda}}%
\global\long\def\Q{\mathcal{Q}}%
\global\long\def\endoofvlambda{\normxilambda\otimes\check{\xi}_{\lambda}^{\mathrm{norm}}}%
\global\long\def\cntwok{\left(\C^{n}\right)^{\otimes2k}}%
\global\long\def\vsigmas{v_{\sigma(1)}\otimes\dots\otimes v_{\sigma(k)}\otimes v_{\sigma'(1)}\otimes\dots\otimes v_{\sigma'(k)}}%
\global\long\def\D{\mathcal{D}}%
\global\long\def\tablambda{\mathrm{Tab}_{\mathrm{\normxilambda}}\left(\lambda^{+}(2k)\right)}%
\global\long\def\part{\mathrm{Part}\left([2k]\right)}%
\global\long\def\ei{e_{i_{i}}\otimes\dots\otimes e_{i_{k}}}%
\global\long\def\csk{\C\left[S_{k}\right]}%
\global\long\def\mixedtensor{\cnk\otimes\left(\left(\C^{n}\right)^{\vee}\right)^{\otimes l}}%
\global\long\def\shortmixedtensor{C_{k,l}^{n}}%
\global\long\def\f{F_{r}}%
\global\long\def\b{\mathcal{B}_{r}^{\times}}%
\global\long\def\u{\mathcal{U}_{\lambda,n}}%
\global\long\def\refinedpart{\overset{\star}{\mathrm{Part}}\left(\left[|w|_{f}k\right]\right)}%
\global\long\def\skpart{\mathrm{Part}_{\leq S_{k}}\left([2k]\right)}%
\global\long\def\gsigma{G\left(\sigma_{f},\tau_{f},\pi_{i}\right)}%
\global\long\def\hatgsigma{\hat{G}\left(\sigma_{f},\tau_{f},\pi_{i}\right)}%
\global\long\def\gamsig{\Gamma\left(\sigma_{f},\tau_{f},\pi_{i}\right)}%
\global\long\def\tgamsig{\tilde{\Gamma}\left(\sigma_{f},\tau_{f},\pi_{i}\right)}%
\global\long\def\Astack{\mathcal{A}_{\hat{\Lambda}'}}%
\global\long\def\B{\mathcal{B}_{\hat{\Lambda}'}}%

\title{Random permutations acting on $k$--tuples have near--optimal spectral
gap for $k=\mathrm{poly}(n)$ }
\author{Ewan Cassidy}
\maketitle
\begin{abstract}
We extend Friedman's theorem to show that, for any fixed $r>1$, a
random $2r$--regular Schreier graph associated with the action of
$r$ uniformly random permutations of $[n]$ on $k_{n}$--tuples
of distinct elements in $[n]$ has a near--optimal spectral gap with
high probability, provided $k_{n}\leq n^{\frac{1}{20}-\epsilon}.$
Previously this was known only for $k$--tuples where $k$ is fixed.
In fact, we prove the stronger result of strong convergence of random
permutations in irreducible representations of quasi--exponential
dimension.

Along the way, we give a new bound for the expected stable irreducible
character of a random permutation obtained via a word map, showing
that $\mathbb{E}\left[\chi^{\mu}\left(w(\sigma_{1},\dots,\sigma_{r})\right)\right]=O\left(\frac{1}{\dim\chi^{\mu}}\right)=O\left(n^{-k}\right)$,
where $k$ is the number of boxes outside the first row of the Young
diagram $\mu,$ solving one aspect of a conjecture of Hanany and Puder.
We obtain this bound using an extension of Wise's $w$--cycle conjecture.
\end{abstract}
\tableofcontents{}

\section{Introduction\label{sec:Introduction}}

\subsection{Random regular graphs and word maps}

Let $G_{n,r}$ be a connected $2r$--regular graph on $n$ vertices,
$A_{G_{n,r}}$ its adjacency matrix with eigenvalues $2r=\lambda_{1}>\lambda_{2}\geq\dots\geq\lambda_{n}\ge-2r$.
Denote $\lambda(G_{n,r})=\mathrm{max}_{|\lambda_{i}|\neq2r}|\lambda_{i}|$
the largest non--trivial eigenvalue. The Alon--Boppana bound dictates
that 
\[
\lambda\left(G_{n,r}\right)\geq2\sqrt{2r-1}-o(1),
\]
so that $2r$--regular graphs with $\lambda\left(G_{n,r}\right)\leq2\sqrt{2r-1}$
have optimal spectral gaps, see \cite{Alon86,Friedman03,Nilli91}
for example.

In \cite{Alon86}, Alon conjectured that a random $2r$--regular
graph on $n$ vertices, $G_{n,r},$ has a near--optimal spectral
gap (i.e. $\lambda\left(G_{n,r}\right)\le2\sqrt{2r-1}+\epsilon$)
with high probability (meaning with probability $\to1$ as $n\to\infty$).
Alon's conjecture was resolved by Friedman \cite{Friedman2008}.
\begin{thm}[Friedman]
If $G_{n,r}$ is a random $2r$--regular graph on $n$ vertices,
then for any $\epsilon>0,$
\[
\mathbb{P}\left[\lambda(G_{n,r})\le2\sqrt{2r-1}+\epsilon\right]\overset{n\to\infty}{\longrightarrow}1.
\]
 
\end{thm}

To sample a random $2r$--regular graph Friedman uses the permutation
model, taking $G_{n,r}$ to be a random Schreier graph associated
to the action of $r$ uniformly random permutations on $[n]$, denoted
$\mathrm{Sch}\left(S_{n}\curvearrowright[n],\sigma_{1},\dots,\sigma_{r}\right)$,
see \S\ref{subsec:Schreier-Graphs}. Questions remain about the expansion
properties of Cayley graphs of $S_{n}.$ These questions are very
different in the sense that the Cayley graphs have $n!$ vertices
rather than $n$, and to answer them, one must have some control over
\emph{all }non--trivial finite dimensional complex representations
of $S_{n}.$ Kassabov, in \cite{Kassabov}, proved that explicit generating
sets of a bounded size can be constructed such that the corresponding
Cayley graphs of $S_{n}$ form a family of bounded--degree expanders.
\begin{thm}[Kassabov]
There exist $L,\epsilon>0$ such that, for any $n$, there exists
an explicit generating set $X_{n}$ of $S_{n}$, of size at most $L$,
such that the Cayley graphs $\bigg\{\mathrm{Cay}\left(S_{n},X_{n}\right)\bigg\}_{n\geq1}$
form a family of $\epsilon$--expanders.
\end{thm}

Whether or not this is typical behaviour, that is, whether or not
Cayley graphs of $S_{n}$ with a random generating set of fixed size
have a uniform spectral gap with high probability, remains an open
problem. Our main result below is in this direction. For any integers
$r,k>0,$ define $\mathcal{G}_{r}\left(n,k\right)$ to be the collection
of Schreier graphs $\mathrm{Sch}\left(S_{n}\curvearrowright[n]_{k},\sigma_{1},\dots,\sigma_{r}\right),$
where $[n]_{k}$ denotes the set of all $k$--tuples of distinct
elements in $[n]$ and $\sigma\left(i_{1},\dots,i_{k}\right)=\left(\sigma(i_{1}),\dots,\sigma(i_{k})\right).$ 
\begin{thm}
\label{thm:main theorem Graph}Fix any integer $r>1$ and let $\alpha<\frac{1}{20}.$
For any sequence of positive integers $\left(k_{n}\right)_{n\geq1}$
with $1\le k_{n}\leq n^{\alpha}$, let $\left(G_{n,r,k_{n}}\right)_{n\geq1}$
be a sequence of random $2r$--regular Schreier graphs, where for
each $n\geq1$, $G_{n,r,k_{n}}\in\mathcal{G}_{r}\left(n,k_{n}\right)$
is obtained by choosing $r$ i.i.d. uniformly random permutations,
$\sigma_{1},\dots,\sigma_{r}\in S_{n}.$ 

Then, for any $\epsilon>0,$
\[
\mathbb{P}\left[\lambda\left(G_{n,r,k_{n}}\right)\leq2\sqrt{2r-1}+\epsilon\right]\overset{n\rightarrow\infty}{\longrightarrow}1.
\]
\end{thm}

One sees that $[n]_{n}=S_{n}$ and so, in this case, a randomly sampled
Schreier graph is simply a random Cayley graph of $S_{n}$ of degree
$2r$. Thus, the above result for $\alpha=1$ ($k_{n}=n)$ would imply
that, for each fixed $r>1,$ not only are random $2r$--regular Cayley
graphs of $S_{n}$ uniform expanders with high probability, but that
they have a near--optimal spectral gap with high probability. As
we mentioned above, this question is still wide open, but we include
this discussion for motivation. 

Previously, Theorem \ref{thm:main theorem Graph} was known for tuples
of size $1$ (i.e. with $\alpha=0)$, which is exactly Friedman's
theorem \cite{Friedman2008} (see also the shorter proof given by
Bordenave \cite{Bordenave2020}); for pairs of distinct elements by
Bordenave and Collins \cite{BordenaveCollins2018} and, more recently,
for $k$--tuples for any \emph{fixed} $k$, by Chen, Garza--Vargas,
Tropp and van Handel in \cite{CGVTvH2024}. One can then consider
Theorem \ref{thm:main theorem Graph} as progress towards bridging
the substantial gap between these results and the open questions regarding
the spectral gap of random fixed--degree Cayley graphs of $S_{n}$.
Indeed, we now know that the random regular Schreier graphs described
above have a near--optimal spectral gap with high probability when
the size $k$ of the tuple is allowed to grow with $n$, and we can
even take $k$ to grow \emph{polynomially }with $n$. 
\begin{rem}
Prior to the result for fixed $k$ in \cite{CGVTvH2024}, it was previously
known that for any fixed $k,$ a random $2r$--regular Schreier graph
of $S_{n}$ has a uniform spectral gap with high probability. Indeed,
in \cite[Theorem 2.1]{Friedmanetal}, they prove that, for any fixed
$k>0$ and $r>1$, if $G_{n,r,k}\in\mathcal{G}_{r}\left(n,k\right)$
is a random $2r$--regular Schreier graph, then for any $\epsilon>0,$
\[
\mathbb{P}\left[\lambda\left(G_{n,r,k}\right)\leq(1+\epsilon)2r\left(\frac{\sqrt{2r-1}}{r}\right)^{\frac{1}{r+1}}\right]\overset{n\to\infty}{\longrightarrow}1.
\]
\end{rem}

\begin{rem}
\label{rem: ramon remark 1}A comment on `derandomization'. In \cite{Mohanty},
using Bordenave's proof of Friedman's theorem, they give a construction
of a random regular graph on $O(N)$ vertices, that has a near--optimal
spectral gap with high probability (probability $\geq1-O\left(\frac{\log N}{2^{100\left(\log N\right)^{\frac{1}{4}}}}\right)$).
They obtain such a graph by randomly choosing a `seed' $s\in\{0,1\}^{O\left(\log(N)^{2}\right)}$
and using this as the input to an algorithm. That is, their construction
requires $\sim\log(N)^{2}$ random bits. With the construction in
Theorem \ref{thm:main theorem Graph}, one can construct a regular
graph on $N$ vertices with a near--optimal spectral gap with high
probability (probability $\geq1-O\left(\frac{1}{\left(\log\left(N\right)^{20}\right)^{1-(20+\epsilon)\alpha}}\right)$),
using $\sim\log(N)^{20}$ random bits, by sampling random permutations
$\sigma_{1},\dots,\sigma_{r}\in S_{n}$, with $n\sim\log(N)^{20}$.
Although this requires slightly more random bits than the algorithm
given in \cite{Mohanty}, it is a more algebraic construction and
illustrates a different approach to such a problem. \footnote{We thank Ramon van Handel for pointing out this application. }
\end{rem}

Our proof of Theorem \ref{thm:main theorem Graph} relies on the remarkable
new approach to \emph{strong convergence }detailed in \cite{CGVTvH2024}
referred to as the `polynomial method', see \S\ref{subsec:Strong-Convergence},
as well as the additional criterion for temperedness of \emph{arbitrary}
functions on finitely generated groups $\Gamma$ with a finite fixed
generating set (an adaptation of the classical notion of a tempered
representation) given by Magee and de la Salle in \cite{MAgeeDelaSalle}. 

To use this approach effectively, we prove our other main theorem
pertaining to word maps and stable irreducible characters (see \S\ref{subsec:Word-maps}),
in which we prove a new bound on the expected stable irreducible character
of a $w$--random permutation. Given a Young diagram $\lambda\vdash k$
and $n\geq k+\lambda_{1},$ we define the Young diagram $\lambda^{+}(n)=\left(n-k,\lambda\right)\vdash n$,
see \S\ref{subsec:Symmetric-group}.
\begin{thm}
\label{thm: word map main theorem}Let $w\in F_{r}=\langle x_{1},\dots,x_{r}\rangle$
be a word in the free group on $r$ generators with $w\neq e$. Fix
$k\in\mathbb{Z}_{>0}$ and $\lambda\vdash k$. If $w$ is not a proper
power, then 
\[
\mathbb{E}_{w}\left[\chi^{\lambda^{+}(n)}\right]\overset{\mathrm{def}}{=}\mathop{\mathop{\mathbb{E}}_{\sigma_{1},\dots,\sigma_{r}\in S_{n}}\left[\chi^{\lambda^{+}(n)}\left(w\left(\sigma_{1},\dots,\sigma_{m}\right)\right)\right]}=\left(\frac{1}{\dim\chi^{\lambda^{+}(n)}}\right)=O\left(\frac{1}{n^{k}}\right).
\]
\end{thm}

\begin{rem}
\label{rem: main bound holds for proper powers}That $\mathbb{E}_{w}\left[\chi^{\lambda^{+}(n)}\right]$
can be written as a rational expression in $n$ follows from (\ref{eq: expected character with pi, weingarten and N})
or indeed by combining \cite{Nica,LinialPuder} with \cite[Proposition B.2]{HananyPuder}.
In the case of $w$ being a proper power, the conclusion of Theorem
\ref{thm: word map main theorem} follows from \cite{Nica} and \cite[Section 4]{LinialPuder}. 
\end{rem}

\begin{rem}
We prove Theorem \ref{thm: word map main theorem} using `combinatorial
integration'. We use a projection formula obtained in our previous
paper \cite{Cassidy2023} to translate the computation of $\mathbb{E}_{w}\left[\chi^{\lambda^{+}(n)}\right]$
in to combinatorial problem. We then prove an extension of Wise's
$w$--cycles conjecture (see \cite{LouderWilton,HelferWise,Wise}),
a theorem in low--dimensional topology, to obtain our final bound. 
\end{rem}

Theorem \ref{thm: word map main theorem} generalizes the work of
Nica \cite{Nica} and Linial and Puder \cite{LinialPuder}, where
it is established that, if $\#\mathrm{fix}(\sigma)$ is the number
of fixed points of a permutation $\sigma\in S_{n},$ then, if $w$
is not a proper power,
\begin{equation}
\mathbb{E}_{w}\left[\#\mathrm{fix}(\sigma)-1\right]=O\left(\frac{1}{n}\right).\label{eq: expression for E (fix -1)}
\end{equation}
Of course, $\#\mathrm{fix}-1$ is the character of the $(n-1)$--dimensional
standard representation of $S_{n},$ $\chi^{(n-1,1)}$, so that this
is a special case of our Theorem \ref{thm: word map main theorem}.\footnote{It is exactly this result of Nica that is used in \cite{CGVTvH2024}
to give a substantially shorter proof of Friedman's theorem.} In \cite{PuderParzanchewski2012}, Puder and Parzanchevski prove
that this decay is actually much faster depending on an algebraic
property of the fixed word $w.$ They show that 
\[
\mathbb{E}_{w}\left[\#\mathrm{fix}(\sigma)-1\right]=O\left(\frac{1}{n^{\pi(w)-1}}\right)
\]
where $\pi(w)$ is the \emph{primitivity rank }of the word $w$ (see,
for example \cite[Definition 1.7]{PuderParzanchewski2012}\emph{,
}for the definition). Hanany and Puder \cite[Theorem 1.3]{HananyPuder}
generalize this result to other stable irreducible representations
of $S_{n}.$ They show that, if $w$ is not a proper power, then for
any $k\neq0,1$ and for any $\lambda\vdash k,$ 
\[
\mathbb{E}_{w}\left[\chi^{\lambda^{+}(n)}\right]=O\left(\frac{1}{n^{\pi(w)}}\right).
\]
 Theorem \ref{thm: word map main theorem} is thus an improvement
on this bound in the regime $k>\pi(w).$ Perhaps more importantly
though, it gives a bound that \emph{improves }when considering stable
irreducible representations of higher dimensions. 

In \cite[Conjecture 1.8]{HananyPuder}, it is conjectured that 
\[
\mathbb{E}_{w}\left[\chi^{\lambda^{+}(n)}\right]=O\left(\frac{1}{\left(\dim\chi^{\lambda^{+}(n)}\right)^{\pi(w)-1}}\right)=O\left(\frac{1}{n^{k(\pi(w)-1)}}\right).
\]
So Theorem \ref{thm: word map main theorem} confirms the conjecture
for words $w$ with primitivity rank $\pi(w)=2.$ In cases where there
is no further information on $w$ other than $w\neq e$ and $w$ is
not a proper power, the bound given in Theorem \ref{thm: word map main theorem}
is expected to be sharp. In fact, it has recently been conjectured
by Puder and Shomroni (see \cite[Conjecture 1.2]{PuderShomroni})
that we should have 
\[
\mathbb{E}_{w}\left[\chi^{\lambda^{+}(n)}\right]=O\left(\frac{1}{n^{k(\mathrm{s}\pi(w))}}\right),
\]
where $\mathrm{s}\pi(w)$ is a property of $w$ known as the \emph{stable
primitvity rank }(this was first introduced by Wilton in \cite{Wilton},
indeed see Definition 10.6 in (\emph{ibid.}) or \cite[Definition 4.1]{PuderShomroni}
for an equivalent definition). They also conjecture that this bound
is the best one possible. That is, for every word $w\in F_{r}$, for
every $k,$ there is a Young diagram $\lambda\vdash k$ such that
$\frac{1}{n^{k(\mathrm{s}\pi(w))}}=O\left(\mathbb{E}_{w}\left[\chi^{\lambda^{+}(n)}\right]\right).$
An additional conjecture of Wilton (see e.g. \cite[Conjecture 4.7]{PuderShomroni})
posits that, for any $w\in F_{r},$ we have the equality $\mathrm{s}\pi(w)=\pi(w)-1$,
so that the conjecture of Hanany and Puder is a combination of these
two conjectures. 

\subsection{Strong Convergence\label{subsec:Strong-Convergence}}

As we noted earlier, we actually prove a stronger statement than Theorem
\ref{thm:main theorem Graph}, that of strong convergence of random
\emph{permutation} representations of $F_{r}=\left\langle x_{1},\dots,x_{r}\right\rangle $. 
\begin{defn}
\label{def: strong convergence}Given a sequence of random, finite
dimensional, unitary representations $\{\pi_{n}:F_{r}\to U\left(N_{n}\right)\}_{n\geq1},$
we say that $\pi_{n}$ strongly converge to the left regular representation
$\lambda:F_{r}\to U\left(\ell^{2}\left(F_{r}\right)\right)$ asymptotically
almost surely (a.a.s.) if, for any $z\in\C\left[F_{r}\right]$, for
any $\epsilon>0$, we have 
\begin{equation}
\mathbb{P}\left[\big|\left\Vert \pi_{n}(z)\right\Vert -\left\Vert \lambda(z)\right\Vert \big|<\epsilon\right]\overset{n\to\infty}{\longrightarrow}1.\label{eq: strong convergence}
\end{equation}
The norms here are operator norms, see \S\ref{subsec: C* algebras of free groups}. 
\end{defn}

We are interested in the cases whereby the representations $\pi_{n}$
factor as $\pi_{n}=\rho\circ\phi_{n},$ where $\phi_{n}\in\hom\left(F_{r},S_{n}\right)$
is random and $\rho:S_{n}\to\mathrm{End}\left(V\right)$ is a representation
of $S_{n}$ of dimension $N_{n}.$ The power of strong convergence
in this setting is that, since the convergence must hold for \emph{every}
$z\in\C\left[F_{r}\right]$, it allows us to prove results like Theorem
\ref{thm:main theorem Graph}, which require (\ref{eq: strong convergence})
to hold only for specific elements of $\C\left[F_{r}\right].$

For example, Friedman's Theorem combined with the Alon--Boppana bound
can be written 
\begin{equation}
\mathbb{P}\left[\big|\left\Vert \pi_{n}(z)\right\Vert -2\sqrt{2r-1}\big|<\epsilon\right]\overset{n\to\infty}{\longrightarrow}1\label{eq: convergence of adjacency matrix, standard rep}
\end{equation}
 where 
\begin{equation}
z=x_{1}+x_{1}^{-1}+\dots+x_{r}+x_{r}^{-1}\label{eq: adjacency matrix graph}
\end{equation}
 and $\pi_{n}=\rho^{(n-1,1)}\circ\phi_{n}$, with $\phi_{n}\in\hom\left(F_{r},S_{n}\right)$
uniformly random (i.e. obtained by uniformly randomly choosing a permutation
$\sigma_{i}\in S_{n}$ for each $i$ and mapping $x_{i}\mapsto\sigma_{i}$)
and where 
\[
\rho^{(n-1,1)}:\mathbb{C}\left[S_{n}\right]\to\mathrm{End}\left(V^{(n-1,1)}\right)
\]
 is the standard representation of $S_{n}.$ This is because, if $\rho:S_{n}\to\mathrm{End}\left(\C^{n}\right)$
is the defining representation of $S_{n},$ then
\[
\rho\left(\sigma_{1}+\sigma_{1}^{-1}+\dots+\sigma_{r}+\sigma_{r}^{-1}\right)
\]
is the adjacency matrix of the graph $\mathrm{Sch}\left(S_{n}\curvearrowright[n],\sigma_{1},\dots,\sigma_{r}\right)$
and $2\sqrt{2r-1}=\|\lambda(z)\|$ in this case. We consider the orthogonal
complement to the $S_{n}$--invariant vectors in $\C^{n}$ to account
for the trivial eigenvalue $2r.$ In the case of $\rho$ above, that
is the standard representation. 

More generally, for any $k\in\mathbb{Z}_{>0}$ we define 
\[
\bar{\rho}_{n,k}:\mathbb{C}\left[S_{n}\right]\to\mathrm{End}\left(\cnk\right),
\]
 the $k^{\mathrm{th}}$ tensor power of the defining representation.
We then define 
\[
\rho_{n,k}:\mathbb{C}\left[S_{n}\right]\to\mathrm{End}\left(\cnk_{\perp\mathrm{triv}}\right)
\]
to be the restriction of $\bar{\rho}_{n,k}$ to the orthogonal complement
to the $S_{n}$--invariant vectors, and define a random sequence
of unitary representations of $F_{r},$ $\{\pi_{n,k}\overset{\mathrm{def}}{=}\rho_{n,k}\circ\phi_{n}\}_{n\ge1},$
with $\phi_{n}\in\hom\left(F_{r},S_{n}\right)$ uniformly random as
before. 

In this language, Bordenave and Collins \cite{BordenaveCollins2018}
proved strong convergence to the left regular representation $\lambda:F_{r}\to U\left(\ell^{2}\left(F_{r}\right)\right)$
a.a.s. for the sequence of random representations of $F_{r}$, $\{\pi_{n,1}\}_{n\ge1}$.
They extend their results to show strong convergence a.a.s. for $\{\pi_{n,2}\}_{n\geq1}$.
With the new `polynomial method' developed in \cite{CGVTvH2024},
strong convergence a.a.s. for $\{\pi_{n,k}\}_{n\geq1}$ for any fixed
$k$ is now known. We extend their method and combine with analytic
results of Magee and de la Salle \cite{MAgeeDelaSalle} to prove our
final main theorem below.
\begin{thm}
\label{thm:Main theorem strong convergence}For any $\alpha<\frac{1}{20}$,
for any $z\in\C\left[F_{r}\right]$ and for any $\epsilon>0,$ 
\[
\mathbb{P}\left[\sup_{k_{n}\leq n^{\alpha}}\bigg|\left\Vert \pi_{n,k_{n}}(z)\right\Vert -\left\Vert \lambda(z)\right\Vert \bigg|<\epsilon\right]\overset{n\to\infty}{\longrightarrow}1.
\]
 
\end{thm}

A simple corollary of the above theorem is that one can establish
the strong convergence a.a.s. to the left regular representation of
$\rho\circ\phi_{n}$, for any non--trivial irreducible representation
$\rho$ of $S_{n},$ provided $\dim\rho$ grows at most quasi--exponentially
in $n$. This is a vast improvement on the previous results in \cite{CGVTvH2024},
whereby this fact is established for $\dim\rho$ growing at most polynomially
in $n$ (for polynomials of any fixed degree). 
\begin{cor}
\label{cor: strong convergence corollar}For any $\alpha<\frac{1}{20}$,
for any $z\in\C\left[F_{r}\right]$ and for any $\epsilon>0,$ 
\[
\mathbb{P}\left[\sup_{\overset{\rho}{{\scriptscriptstyle \dim\rho\le C\exp\left(n^{\alpha}\right)}}}\bigg|\left\Vert \rho\circ\phi_{n}(z)\right\Vert -\left\Vert \lambda(z)\right\Vert \bigg|<\epsilon\right]\overset{n\to\infty}{\longrightarrow}1.
\]
where the supremum is over all non--trivial irreducible representations
of $S_{n}$.
\end{cor}

\begin{proof}
See the proof in \cite[Theorem 5.8]{vanHandelSURVEY}.
\end{proof}
\begin{rem}
Theorem \ref{thm:main theorem Graph} follows immediately from Theorem
\ref{thm:Main theorem strong convergence} by taking $z=x_{1}+x_{1}^{-1}+\dots+x_{r}+x_{r}^{-1}.$
Theorem \ref{thm:Main theorem strong convergence} is a much stronger
statement, and it is certainly not necessary prove such a statement
to prove Theorem \ref{thm:main theorem Graph}.
\end{rem}

\subsection{Overview of paper}

\S\ref{sec:Background} gives background information on variations
of Schur--Weyl duality, including a refinement of Schur--Weyl--Jones
duality and a projection formula given by the author in \cite{Cassidy2023}.
We also give a brief overview of the Weingarten calculus for the symmetric
group and a number of necessary results from \cite{MAgeeDelaSalle}.

\S\ref{sec: Proof of strong convergence} contains our proof of Theorem
\ref{thm:Main theorem strong convergence}, assuming Theorem \ref{thm: word map main theorem}
and Theorem \ref{thm: word map polynomial theorem}. The proofs of
Theorem \ref{thm: word map main theorem} and Theorem \ref{thm: word map polynomial theorem}
are given in \S\ref{sec: proof of word map theorem}. 

\subsection*{Acknowledgements}

This paper is part of a project that received funding from the European
Research Council (ERC) under the European Union\textquoteright s Horizon
2020 research and innovation programme (grant agreement No 949143). 

We thank Michael Magee for the many conversations regarding this paper,
as well as Doron Puder, Lars Louder and Ramon van Handel for their
valuable input.

\section{Background\label{sec:Background}}

\subsection*{A note on notation}

We will use the following Vinogradov notation. Given two functions
$f,g$ of $n\in\mathbb{Z}_{>0},$ we will write $f\ll g$ to indicate
that that there exist constants $C$ and $N$ such that $f(n)\leq Cg(n)$
for all $n\geq N.$ We will then write $f=O(g)$ to mean that $f\ll|g|$
and $f\asymp g$ to mean that $f\ll g$ \emph{and }$f\ll g.$ If the
implied constants depend on some additional parameter, then these
will be added as subscripts. 

\subsection{Symmetric group\label{subsec:Symmetric-group}}

We will denote by $S_{n}$ the symmetric group on $n$ elements and
view this as the group of permutations of $[n]\overset{\mathrm{def}}{=}\{1,\dots,n\}$. 

A \emph{Young diagram }(YD) $\mu=\left(\mu_{1},\dots,\mu_{l(\mu)}\right)\vdash n$
is an arrangement of boxes in to aligned rows, in which row $i$ contains
$\mu_{i}$ boxes. By definition, we have $\mu_{1}\geq\mu_{2}\geq\dots\geq\mu_{l(\mu)}>0,$
so that the number of boxes in each row is non--increasing as the
row index increases. We also have $\mu_{1}+\dots+\mu_{l(\mu)}=|\mu|=n.$

The equivalence classes of finite dimensional, complex irreducible
representations of $S_{n}$ are in a natural bijection with the set
of Young diagrams $\mu\vdash n,$ see \cite{VershikOkounkov} for
example. We will denote the representation corresponding to $\mu$
by $V^{\mu}$ and denote its character by $\chi^{\mu}.$ Every finite--dimensional
complex representation $V$ of $S_{n}$ admits a decomposition 
\[
V\cong\bigoplus_{\mu\vdash n}\left(V^{\mu}\right)^{\oplus a_{\mu}},
\]
unique up to isomorphism (i.e. unique up to isomorphisms that preserve
each component $\left(V^{\mu}\right)^{\oplus a_{\mu}}$ of the decomposition).
For each $\mu\vdash n,$ the subspace $\left(V^{\mu}\right)^{\oplus a_{\mu}}$
is called the $V^{\mu}$--isotypic component of $V.$ 

We will denote by $d_{\mu}\overset{\mathrm{def}}{=}\chi^{\mu}(e)=\mathrm{dim}\left(V^{\mu}\right)$$.$
Given a representation $V$ of $S_{n}$ and a permutation $\sigma,$
we may write $V(\sigma)$ or just $\sigma$ to mean the element of
$\mathrm{End}\left(V\right)$ corresponding to $\sigma,$ defined
by the representation. Given $\lambda=(\lambda_{1},\dots,\lambda_{l(\lambda)}),$
, we define $V^{\lambda^{+}(n)}$ to be the irreducible representation
of $S_{n}$ (when $n-|\lambda|>\lambda_{1}$) corresponding to the
YD 
\[
\lambda^{+}(n)=\left(n-|\lambda|,\lambda_{1},\dots,\lambda_{l(\lambda)}\right)\vdash n.
\]

\subsection{Partition algebra\label{subsec:Partition-algebra}}

The \emph{partition algebra }appeared in the work of Martin \cite{Martin1994},
Martin and Saleur \cite{MartinSaleur1992} and Jones \cite{Jones}
as a generalization of the Temperley--Lieb algebra and the Potts
model in statistical mechanics. It is closely connected to the symmetric
group via Schur--Weyl--Jones duality, which was first detailed explicitly
by Jones in \cite{Jones}.

\subsubsection{Set partitions}

A set partition of the set $[m]$ is a grouping of the elements in
to non--empty, disjoint subsets in which every element is in exactly
one subset. We will write $i\sim j$ to indicate that $i$ and $j$
are in the same subset of a partition, and refer to the subsets as
\emph{blocks}. The number of blocks of a partition $\pi\in\mathrm{Part}\left([m]\right)$
will be denoted $|\pi|.$ Obviously, $1\leq|\pi|\leq m$. 

We denote by $\mathrm{Part}\left([m]\right)$ the set of set partitions
of $[m]$. One may define a natural partial ordering on $\mathrm{Part}\left([m]\right)$
by refinement: $\pi_{1}\leq\pi_{2}$ if every block of $\pi_{1}$
is contained in a block of $\pi_{2}.$ We say that $\pi_{1}$ is finer
than $\pi_{2}$ or that $\pi_{2}$ is coarser than $\pi_{1}.$ From
this partial ordering one may derive the M{\"o}bius function $\mu$
on $\mathrm{Part}\left([m]\right)$. If $\pi_{2}=\Big\{\mathcal{P}_{1},\dots,\mathcal{P}_{|\pi_{2}|}\Big\}$
and $\pi_{1}$ is a refinement of $\pi_{2}$ in which each block $\mathcal{P}_{i}$
is split in to $b_{i}$ subsets, then 
\[
\mu\left(\pi_{1},\pi_{2}\right)=(-1)^{|\pi_{1}|-|\pi_{2}|}\prod_{i=1}^{|\pi_{2}|}(b_{i}-1)!,
\]
as is described in the foundational paper by Rota \cite{Rota}. 

Given partitions $\pi_{1},\pi_{2}\in\mathrm{Part}\left([m]\right),$
their \emph{meet }$\pi_{1}\wedge\pi_{2}$ is defined by $i\sim j$
in $\pi_{1}\wedge\pi_{2}$$\iff i\sim j$ in both $\pi_{1}$ and $\pi_{2}.$
Their \emph{join $\pi_{1}\vee\pi_{2}$ }is defined by $i\sim j$ in
$\pi_{1}\vee\pi_{2}$$\iff i\sim j$ in either $\pi_{1}$ or $\pi_{2}.$ 

\subsubsection{Partition algebra}

Given $k\in\mathbb{Z}_{>0}$ and $n\in\C$, the partition algebra
$\P$ is an associative unital algebra over $\C.$ It is the vector
space 
\[
\big\langle\pi:\ \pi\in\part\big\rangle_{\C},
\]
with a natural multiplication as described below

To each $\pi\in\part$ we associate a diagram consisting of two rows
of $k$ vertices labeled from $1,\dots k$ on the top row and from
$k+1,\dots,2k$ on the bottom row. We draw an edge between any two
vertices $i$ and $j$ if and only if $i\sim j$ in $\pi.$ Any diagram
consisting of two rows of $k$ vertices with all connected components
complete clearly corresponds to an element $\pi\in\part.$ The product
$\pi_{1}\pi_{2}$ is obtained by merging the bottom row of vertices
of $\pi_{1}$ with the top row of vertices of $\pi_{2}$ to obtain
a new diagram $\tilde{\pi}_{3}$ with three rows of $k$ vertices;
adding edges between any vertices that are in the same connected component
(where there is not already an edge); removing the middle row of vertices
and any adjacent edges to obtain $\pi_{3}\in\part.$ Then we have
\[
\pi_{1}\pi_{2}=n^{\gamma}\pi_{3},
\]
 where $\gamma$ is the number of connected components of $\tilde{\pi}_{3}$
with vertices solely in the middle row. 

Martin and Saleur \cite{MartinSaleur1992} describe some important
properties of $\P.$
\begin{thm}
\label{thm: partition algebra is semisimple}The partition algebra
$\P$ is semisimple for any $n\in\C,$ unless $n$ is an integer in
the range $[0,2k-1].$
\end{thm}

\begin{thm}
\label{thm: irreps of partition algebra}If $\P$ is semisimple, then
there is a natural bijection between the equivalence classes of simple
modules over $\P$ and the set 
\[
\Lambda_{k,n}\overset{\mathrm{def}}{=}\{\lambda\vdash i:\ i=0,\dots,k\}.
\]
 
\end{thm}

By convention, there is exactly one Young diagram $\lambda\vdash0.$
We will write $R^{\lambda}$ to denote the equivalence class corresponding
to $\lambda.$

\subsection{Schur--Weyl--Jones duality\label{subsec:SchurWeylJones-duality}}

There is an analogue of classical Schur--Weyl duality due to Jones
\cite{Jones}, which we refer to as Schur--Weyl--Jones duality. 

There is a natural right action of $\P$ on the vector space $\cnk.$
For any standard orthonormal basis vectors $\ei$ and $e_{i_{k+1}}\otimes\dots\otimes e_{i_{2k}}$
and any $\pi\in\part,$ we have 
\[
\big\langle\left(\ei\right)\pi,\ e_{i_{k+1}}\otimes\dots\otimes e_{i_{2k}}\big\rangle=\begin{cases}
1 & \mathrm{if}\ j\sim l\mathrm{\ in}\ \pi\iff i_{j}=i_{l}\\
0 & \mathrm{otherwise.}
\end{cases}
\]
This extends linearly to define a right action on $\cnk.$ For each
$\pi\in\P,$ we define $p_{\pi}\in\mathrm{End}\left(\cnk\right),$
where $p_{\pi}(v)=v\pi.$

The symmetric group $S_{n}$ acts on $\cnk$ on the left via the diagonal
permutation action, i.e. for $g\in S_{n}$ and basis vector $\ei\in\cnk,$
\[
g\left(\ei\right)=e_{g\left(i_{1}\right)}\otimes\dots\otimes e_{g\left(i_{k}\right)}.
\]
 Schur--Weyl--Jones duality asserts the following theorem. 
\begin{thm}
\label{thm: schur weyl jones duality}When $S_{n}$ and $\P$ act
as above and $n\geq2k$, 
\begin{enumerate}
\item $\P$ generates $\mathrm{End}_{S_{n}}\left(\cnk\right)$,
\item $\C\left[S_{n}\right]$ generates $\mathrm{End}_{\P}\left(\cnk\right)$.
\item As a $\left(S_{n},\P\right)$--bimodule, 
\[
\cnk\cong\bigoplus_{\lambda\in\Lambda_{k,n}}V^{\lambda^{+}(n)}\otimes R^{\lambda}.
\]
\end{enumerate}
\end{thm}

\subsection{Refinement}

We now present a refinement of Schur--Weyl--Jones duality due to
Sam and Snowden \cite[Section 6.1.3]{SamSnowden}. We will assume
that $k\in\mathbb{Z}_{>0}$ and $n\geq2k$. 

There is an algebra inclusion 
\[
\iota:\csk\to\P,
\]
where, for each $\sigma\in S_{k},$ 
\[
\iota\left(\sigma\right)=\big\{\{1,k+\sigma^{-1}(1)\},\dots,\{k,k+\sigma^{-1}(k)\}\big\}.
\]
 There is also a corresponding restriction map 
\[
R:\P\to\csk
\]
 in the sense that $R\circ\iota=\mathrm{Id}:\csk\to\csk$ with $R(\pi)=0$
for any set partition $\pi\in\part$ that is \emph{not }the image
of some permutation in $S_{k}.$

Notice that the image of $\iota$ is the complex linear span of all
partitions that contain exactly one element $i$ with $0\leq i\leq k$
and exactly one element $j$ satisfying $k+1\leq j\leq2k$ (equivalently,
all partitions for which the corresponding diagram consists of $k$
connected components, each containing exactly one vertex from each
row). 
\begin{defn}
\label{def: contraction maps}We define the linear contraction map
$T_{j}:\cnk\to\left(\C^{n}\right)^{\otimes k-1}$ by 
\[
\begin{aligned} & T_{j}\left(\ei\right)\\
= & e_{i_{1}}\otimes\dots\otimes\dot{e}_{i_{j}}\otimes\dots\otimes e_{i_{k}}\\
\overset{\mathrm{def}}{=} & e_{i_{1}}\otimes\dots\otimes e_{i_{j-1}}\otimes e_{i_{j+1}}\otimes\dots\otimes e_{i_{k}}.
\end{aligned}
\]
\end{defn}

\begin{defn}
\label{def: Dkn }The vector subspace $D_{k}(n)\subset\cnk$ is defined
to be 
\[
D_{k}(n)\overset{\mathrm{def}}{=}\dkn_{\C},
\]
the complex linear span of all $k$--tuples of distinct elements. 
\end{defn}

We use the contraction maps and $D_{k}(n)$ to construct a subspace
on which we have a version of Schur--Weyl duality involving only
the symmetric groups. 
\begin{defn}
\label{def: Akn}The vector subspace $\A\subseteq\cnk$ is defined
by 
\[
\A=D_{k}(n)\cap\bigcap_{j=1}^{k}\ker\left(T_{j}\right).
\]
\end{defn}

\begin{lem}
\label{lem: Akn is kernel of all of Ikn}We have 
\[
\A=\bigcap_{\pi\in\ker\left(R\right)}\ker\left(\pi\right).
\]
 
\end{lem}

Which is to say that, for $\pi\in\part$ and $v\in\A,$ we have $p_{\pi}(v)=0$
unless $\pi=\iota(\sigma)$ for some $\sigma\in S_{k}.$ Moreover,
if $\sigma\in S_{k},$ then $p_{\iota(\sigma)}\in\mathrm{End}\left(\cnk\right)$
permutes the tensor coordinates according to $\sigma.$ That is, 
\[
p_{\iota(\sigma)}\left(\ei\right)=e_{i_{\sigma(1)}}\otimes\dots\otimes e_{i_{\sigma(k)}}.
\]
 This extends linearly to define an action of $\csk$ on $\cnk$ (and
suitable subspaces). We then view $\A$ as a representation of $S_{n}\times S_{k},$
where once again $g\in S_{n}$ acts via the diagonal permutation action
and $\sigma\in S_{k}$ acts via $p_{\iota\left(\sigma^{-1}\right)}.$
We denote this representation 
\[
\Delta:\C\left[S_{n}\times S_{k}\right]\to\mathrm{End}\left(\A\right).
\]
The version of Schur--Weyl duality on $\A$ is the following.
\begin{thm}
\label{thm: refinement of Schur weyl jones duality}Where $S_{n}$
and $S_{k}$ act as above, 
\begin{enumerate}
\item $\csk$ generates $\mathrm{End}_{S_{n}}\left(\A\right)$,
\item $\C\left[S_{n}\right]$ generates $\mathrm{End}_{S_{k}}\left(\A\right)$.
\item The decomposition of $\A$ in to irreducible $S_{n}\times S_{k}$
representations is 
\[
\A\cong\bigoplus_{\lambda\vdash k}V^{\lambda^{+}(n)}\otimes V^{\lambda}.
\]
 
\end{enumerate}
\end{thm}

The main result of the author in \cite{Cassidy2023} is a formula
for the orthogonal projection from $\cnk$ to subspaces isomorphic
to the irreducible `blocks' in the decomposition of $\A.$ Define
\[
\xi\overset{\mathrm{def}}{=}\element\in\A.
\]
We denote by $\normxilambda$ the normalised projection of $\xi$
to the $V^{\lambda}$--isotypic subspace of $\A$. This can be shown
to be non--zero, giving the following lemma.
\begin{lem}
\label{lem: construction of irreducible block}For any $\lambda\vdash k$,
the $S_{n}\times S_{k}$ representation generated by $\normxilambda$
is isomorphic to $V^{\lambda^{+}(n)}\otimes V^{\lambda}.$ That is,
\[
\mathcal{U}_{\lambda^{+}(n)}\overset{\mathrm{def}}{=}\big\langle\Delta\left(g,\sigma\right)\left(\normxilambda\right):\ g\in S_{n},\ \sigma\in S_{k}\big\rangle_{\C}\cong V^{\lambda^{+}(n)}\otimes V^{\lambda}.
\]
\end{lem}

This leads to the main theorem and a corollary of the author in \cite{Cassidy2023}
which will be key for the combinatorial integration method used in
this paper. We will write 
\[
\sum_{\pi}^{\le S_{k}}
\]
to indicate that we are summing over set partitions $\pi$ that are
$\le\iota(\tau)$ for some $\tau\in S_{k}.$
\begin{thm}
\label{thm: main theorem of first paper}For any $\lambda\vdash k$,
the orthogonal projection 
\[
\mathcal{Q}_{\lambda,n}:\cnk\to\mathcal{U}_{\lambda^{+}(n)}
\]
 is given by 
\[
\sum_{\pi\in\part}^{\le S_{k}}c(n,k,\lambda,\pi)p_{\pi},
\]
where 
\begin{equation}
c(n,k,\lambda,\pi)=\frac{d_{\lambda^{+}(n)}(-1)^{|\pi|+k}}{(n)_{|\pi|}}\sum_{\overset{\tau\in S_{k}}{{\scriptstyle {\scriptscriptstyle \iota(\tau)\geq\pi}}}}\chi^{\lambda}(\tau).\label{eq: coeff of pi in projection}
\end{equation}
\end{thm}

\begin{cor}
\label{cor: bitrace corollary}For any $\lambda\vdash k$ and any
$g\in S_{n},$ $\chi^{\lambda^{+}(n)}(g)$ is equal to
\[
\begin{aligned} & \frac{1}{d_{\lambda}}\mathrm{tr}_{\mathcal{U}_{\lambda^{+}(n)}}\left(g\right)\\
= & \frac{1}{d_{\lambda}}\mathrm{tr}_{\cnk}\left(g,\mathcal{Q}_{\lambda,n}\right).
\end{aligned}
\]
\end{cor}

If $\pi\leq\iota(\tau)$ for some $\tau\in S_{k},$ then the diagram
corresponding to $\pi$ is obtained from that of $\iota(\tau)$ by
deleting edges. Let $\mathrm{del}\left(\pi\right)$ be the number
of edges deleted (i.e. $\mathrm{del}\left(\pi\right)=|\pi|-k$, so
that if $\pi\leq\iota\left(\tau_{1}\right)$ \emph{and }$\pi\leq\iota\left(\tau_{2}\right),$
then $\mathrm{del}(\pi)$ is independent of whether we count the edges
deleted from $\tau_{1}$ or $\tau_{2}$).
\begin{lem}
\label{rem: leading coeff of pi in projection}If $\pi\leq\iota(\tau)$
for some $\tau\in S_{k},$ then 
\[
c(n,k,\lambda,\pi)=O\left(\frac{1}{n^{\mathrm{del}(\pi)}}\right).
\]
\end{lem}

\begin{proof}
This is immediate from the observation that $d_{\lambda^{+}(n)}=O\left(n^{k}\right)$. 
\end{proof}
\begin{rem}
If $\pi=\iota(\tau)$, then this is $O(1).$ 
\end{rem}

\subsection{The Weingarten calculus\label{subsec:The-Weingarten-calculus}}

An important component of our proof of Theorem \ref{thm: main theorem of first paper}
is the Weingarten calculus. Originally developed by Weingarten \cite{Weingarten}
and Xu \cite{Xu} and then more rigorously by Collins \cite{Collins2003}
and Collins and Sniady \cite{CollinsSniady2006}, the Weingarten calculus
is a method for computing integrals of matrix coefficients using Schur--Weyl
duality analogues. Below is the main formulation for the Weingarten
calculus for the symmetric group, we direct the reader to the survey
\cite[Section 3]{CollinsWeingartenShort} or indeed any of the aforementioned
papers for an in depth discussion. We will denote the matrix coefficients
for $g\in S_{n}$ acting diagonally on $\cnk$ by $g_{ij}.$

Define, for any multi--index $I=\left(i_{1},\dots,i_{m}\right)$
with $i_{1},\dots,i_{m}\in[n]$ and any set partition $\pi\in\mathrm{Part}\left([m]\right)$
$,$ the function $\delta_{\pi},$ where 
\[
\delta_{\pi}\left(I\right)=\begin{cases}
1 & \mathrm{if}\ j\sim l\mathrm{\ in}\ \pi\implies i_{j}=i_{l}\\
0 & \mathrm{otherwise}.
\end{cases}
\]
 
\begin{thm}
\label{thm: weingarten calculus}For any multi--indices $I=\left(i_{1},\dots,i_{m}\right),$
$J=\left(j_{1},\dots,j_{m}\right)\in[n]^{m},$ with respect to the
uniform measure on $S_{n}$, we have 
\[
\mathop{\int}_{S_{n}}g_{i_{1}j_{1}}\dots g_{i_{m}j_{m}}dg=\sum_{\pi_{1},\pi_{2}\in\mathrm{Part}\left([m]\right)}\delta_{\pi_{1}}(I)\delta_{\pi_{2}}(J)\mathrm{Wg}_{n,m}\left(\pi_{1},\pi_{2}\right),
\]
where $\mathrm{Wg}_{n,m}$ is the Weingarten function for $S_{n},$
\[
\mathrm{Wg}_{n,m}\left(\pi_{1},\pi_{2}\right)=\sum_{\pi\leq\pi_{1}\wedge\pi_{2}}\mu\left(\pi,\pi_{1}\right)\mu\left(\pi,\pi_{2}\right)\frac{1}{(n)_{|\pi|}}.
\]
 
\end{thm}

\begin{rem}
\label{rem: leading term for weingarten function}For any $\pi_{1},\pi_{2}\in\mathrm{Part}\left([m]\right),$
$\mathrm{Wg}_{n,m}\left(\pi_{1},\pi_{2}\right)=O\left(\frac{1}{n^{|\pi_{1}\wedge\pi_{2}|}}\right)$.
In particular, $\mathrm{Wg}_{n,m}\left(\pi_{1},\pi_{1}\right)=O\left(\frac{1}{n^{|\pi_{1}|}}\right)$. 
\end{rem}

\subsection{$C^{*}$--algebras of free groups\label{subsec: C* algebras of free groups}}

We refer to \cite{NicaSpeicher} for a thorough introduction. To the
free group $F_{r}$ we associate its group algebra $\C\left[F_{r}\right],$
the collection of elements of the form 
\[
x=\sum_{w\in F_{r}}x(w)w,
\]
 where each $x(w)\in\C$ and only finitely many are non--zero. One
defines a norm on $\C\left[F_{r}\right]$ by 
\[
\|x\|_{C^{*}\left(F_{r}\right)}=\sup\big\{\|\pi(x)\|_{\mathrm{op}}:\ \pi\ \mathrm{a\ representation\ of}\ F_{r}\big\},
\]
where $\|-\|_{\mathrm{op}}$ is the operator norm on $\ell^{2}\left(F_{r}\right)$.
The completion of $\C\left[F_{r}\right]$ in this norm is the $C^{*}$--algebra
of $F_{r},$ denoted $C^{*}\left(F_{r}\right)$. The left regular
representation of $F_{r}$ is denoted $\lambda$ and admits the \emph{reduced
norm }on $\C\left[F_{r}\right]$ by 
\[
\|x\|_{\mathrm{C_{\lambda}^{*}\left(F_{r}\right)}}=\|\lambda\left(x\right)\|_{\mathrm{op}}.
\]
 The completion of $\C\left[F_{r}\right]$ in this norm is the reduced
$C^{*}$--algebra of $F_{r},$ denoted $C_{\lambda}^{*}\left(F_{r}\right).$

One can define a trace $\tau$ on $\C\left[F_{r}\right]$ to be the
map 
\[
\tau(x)=x(e).
\]
This extends continuously to both the $C^{*}$--algebras defined
above and remains a trace in both cases. 

\subsection{Polynomials and random walks\label{subsec: polynomials and random walks}}

We collect some results and definitions on polynomials and random
walks that are required for the analysis in \S\ref{sec: Proof of strong convergence}.
The proofs are readily available in \cite{MAgeeDelaSalle}. See also
\cite{CGVTvH2024} for more background. 
\begin{lem}[{\cite[Lemma 4.2]{MAgeeDelaSalle}}]
\label{lem: bounding sup pf kth derivative of polynomial }For every
polynomial $P$ in one variable with bounded degree, $\deg\left(P\right)\leq D$
and for every integer $k\leq D,$ 
\[
\sup_{\left[0,\frac{1}{2D^{2}}\right]}\left|P^{(k)}\right|\leq\frac{2^{2k+1}D^{4k}}{(2k-1)!!}\sup_{n\geq D^{2}}\left|P\left(\frac{1}{n}\right)\right|,
\]
where $(2k-1)!!=(2k-1)(2k-1)\dots(3)1.$
\end{lem}

\subsubsection{Random walks and free groups}
\begin{defn}
\label{def: reasonable probability measure}We call a symmetric (i.e.
$\mu(g)=\mu\left(g^{-1}\right))$ probability measure $\mu$ on the
free group $F_{r}$ \emph{reasonable }if its support is finite, contains
the identity element and generates $F_{r}.$ 
\end{defn}

We denote by $(g_{n})_{n\geq0}$ the corresponding random walk in
$F_{r}.$ So we write 
\[
g_{n}=s_{1}\dots s_{n}
\]
 where $s_{i}\in F_{r}$ are i.i.d. according to $\mu.$ 

The spectral radius $\rho=\rho(\mu)=\left\Vert \mu\right\Vert _{C_{\lambda}^{*}\left(F_{r}\right)}$
measures how fast the probability that a random walk returns to where
it started decays. Recall that a \emph{proper power }in the free group
is defined to be an element of the form $u^{d}$ for $u\in F_{r}$
and $d\geq2.$ 
\begin{prop}[{\cite[Proposition 6.1]{MAgeeDelaSalle}}]
\label{prop: bound on prob that random walk ends on power}For any
reasonable probability measure $\mu,$ there is a constant $C=C_{\mu}$
such that 
\[
\mathbb{P}\left(g_{n}\ \mathrm{is\ a\ proper\ power}\right)\leq Cn^{5}\rho^{n}.
\]
\end{prop}

We denote the subspace of elements of $\C\left[F_{r}\right]$ supported
in the ball of radius $q$ by $\C_{\leq q}\left[F_{r}\right].$ 

\subsubsection{Temperedness and strong convergence}

The following definition and proposition can be found in \cite[Section 5]{MAgeeDelaSalle}.
Fix a generating set of size $r$ for the free group $F_{r}.$
\begin{defn}
\label{def: tempered}A function $u:F_{r}\to\C$ is called tempered
if 
\[
\limsup_{n\to\infty}|u\left(\left(x^{*}x\right)^{n}\right)|^{\frac{1}{2n}}\leq\|\lambda(x)\|_{\mathrm{op}}
\]
 for every $x\in\C\left(F_{r}\right),$where $\lambda$ is the left
regular representation. 
\end{defn}

The above property is a part of a criterion for strong convergence
of random representations.
\begin{prop}
\label{prop: criterion for strong convergence}Let $u_{n}:F_{r}\to\C$
be functions, $\pi_{n}$ a sequence of random unitary representations
of $F_{r}$ with finite and non--random dimension. Let $\epsilon_{n}>0.$
If the following conditions are satisfied: 
\begin{itemize}
\item $\lim_{n\to\infty}\epsilon_{n}=0,$
\item we have $|\mathbb{E}\mathrm{Tr}\left(\pi_{n}(x)\right)-u_{n}(x)|\leq\epsilon_{n}\exp\left(\frac{q}{\log\left(2+q\right)^{2}}\right)\|x\|_{C^{*}\left(F_{r}\right)}$
for every $q$ and every $x\in\C_{\leq q}\left[F_{r}\right],$ and 
\item $u_{n}$ is tempered and there is a polynomial $P_{n}$ such that,
for every q and every $x\in\C_{\leq q}\left[F_{r}\right],$ $|u_{n}(x)|\leq P_{n}(q)\|x\|_{C^{*}\left(F_{r}\right)}.$ 
\end{itemize}
Then, for every $y\in\C\left[F_{r}\right]$ and for every $\delta>0,$
\[
\mathbb{P}\left[\left\Vert \pi_{n}(y)\right\Vert >\left\Vert \lambda(y)\right\Vert +\delta\right]\leq C\left(y,\delta\right)\epsilon_{n}
\]
for some constant $C(y,\delta).$ In particular,
\begin{equation}
\lim_{n}\mathbb{P}\left[\left\Vert \pi_{n}(y)\right\Vert >\left\Vert \lambda(y)\right\Vert +\delta\right]=0.\label{eq: prop 2.20 lim_n P |-| =00003D0}
\end{equation}
\end{prop}

To prove that a function is tempered, we can use the following proposition
(stated here for free groups only), from \cite[Proposition 6.3]{MAgeeDelaSalle}.
\begin{prop}
\label{prop: criterion for being tempered}Let $u:F_{r}\to\C$ and
assume that, for every reasonable probability measure $\mu$ on $F_{r}$,
if $\left(g_{n}\right)_{n}$ is the associated random walk on $F_{r},$
\[
\limsup_{n\to\infty}\left(\mathbb{E}\left|u\left(g_{n}\right)\right|\right)^{\frac{1}{n}}\leq\rho(\mu).
\]
Then $u$ is tempered. .
\end{prop}

\begin{rem}
That this proposition holds for free groups is a result of Haagerup's
inequality \cite[Lemma 1.5]{Haagerup}, which asserts that free groups
with their standard generating sets have the \emph{rapid decay property. }
\end{rem}

\subsection{Schreier graphs\label{subsec:Schreier-Graphs}}
\begin{defn}
\label{def:Schreier graph}Let $G$ be a finite group acting on a
finite set $X$ and let $g_{1},\dots,g_{r}\in G.$ The Schreier graph
$\mathrm{Sch}\left(S_{n}\curvearrowright X,g_{1},\dots,g_{r}\right)$
is the $2r$--regular graph on $|X|$ vertices, consisting of vertex
set $X$ and, for every vertex $x\in X,$ for each $g_{i},$ we add
an edge between $x$ and $g_{i}x.$ Note that we allow multiple edges
and loops and that the graph is undirected.

By taking $g_{1},\dots,g_{r}\in G$ i.i.d. uniformly at random, one
obtains a random $2r$--regular graph on $|X|$ vertices. 
\end{defn}

\subsection{Word maps\label{subsec:Word-maps}}

Given a word $w\in F_{r}=\left\langle x_{1},\dots,x_{r}\right\rangle $
and a (finite) group $G,$ one obtains a \emph{word map }
\[
w:G^{r}\to G
\]
 by substitutions. For example, if $w=\left[x_{1},x_{2}\right]=x_{1}x_{2}x_{1}^{-1}x_{2}^{-1}\in F_{2}$
and $g,h\in G,$ then $w\left(g,h\right)=ghg^{-1}h^{-1}\in G.$ 

For any finite group $G$ and class function $\chi:G\to\mathbb{R},$
we define 
\[
\mathbb{E}_{w}\left[\chi\right]\overset{\mathrm{def}}{=}\mathop{\mathbb{E}}_{g_{1},\dots,g_{r}\in G}\left[\chi\left(w\left(g_{1},\dots,g_{r}\right)\right)\right].
\]
In this paper, we will only be considering word maps for the symmetric
group $S_{n}.$
\begin{rem}
$\mathbb{E}_{w}\left[\chi\right]$ is equivalent to $\mathop{\mathbb{E}}_{\hom\left(F_{r},S_{n}\right)}\left[\chi\left(\phi_{n}(w)\right)\right]$,
where $\phi_{n}\in\hom(F_{r},S_{n})$ is chosen uniformly randomly
by choosing uniformly random i.i.d. $\sigma_{1},\dots,\sigma_{r}\in S_{n}$
and mapping each $x_{i}\mapsto\sigma_{i},$ as is described in \S\ref{subsec:Strong-Convergence}.
\end{rem}

\section{\label{sec: Proof of strong convergence}Proof of Theorem \ref{thm:Main theorem strong convergence}
from Theorem \ref{thm: word map main theorem}}

The results of this section are somewhat analogous to those presented
in \cite[Section 7]{MAgeeDelaSalle}. Throughout, $C$ will denote
a constant that does not depend on anything, but that may change from
line to line. We require our bound on the expected irreducible stable
character of a $w$--random permutation from Theorem \ref{thm: word map main theorem}.
Fix any integer $K>0$ and define 
\[
\Sigma_{n,K}\overset{\mathrm{def}}{=}\bigoplus_{\lambda\vdash K}V^{\lambda^{+}(n)}.
\]
We view $\Sigma_{n,K}\subseteq\cnk,$ so that 
\begin{equation}
\dim\left(\Sigma_{n,K}\right)\leq n^{K}.\label{eq:dimesnion of SIGMA}
\end{equation}
The following theorem follows immediately by combining Theorem \ref{thm: word map main theorem}
and Proposition \ref{thm: word map polynomial theorem} -- recall
that, for each $\lambda\vdash K$, we can write 
\[
\mathbb{E}_{w}\left[\chi^{\lambda^{+}(n)}\right]=\frac{P_{w,\lambda,q}\left(\frac{1}{n}\right)}{g_{q,k}\left(\frac{1}{n}\right)},
\]
where $\deg\left(P_{w,\lambda,q}\right)\leq3Kq+Kq^{2}$ and that,
as a rational function in $n$, $\mathbb{E}_{w}\left[\chi^{\lambda^{+}(n)}\right]=O(1)$
for proper powers and $\mathbb{E}_{w}\left[\chi^{\lambda^{+}(n)}\right]=O\left(\frac{1}{n^{K}}\right)$
for non--powers. 
\begin{thm}
\label{thm: word map thm rewritten}For every word $w\in F_{r},$
there is a rational function $\phi_{w}\in\mathbb{Q}\left[x\right]$
such that
\begin{enumerate}
\item \label{enu: word map rewritten part 1}For $n\geq K\max\left(l(w),2\right),$
\[
\phi_{w}\left(\frac{1}{n}\right)=\frac{1}{n^{K}}\mathbb{E}_{w}\left[\mathrm{Tr_{\Sigma_{n,K}}}\right].
\]
\item \label{enu:word map rewritten part 2}If $w$ is not the identity
and $l(w)\leq q,$ then $g_{q,K}\phi_{w}$ is a polynomial of degree
$\leq D_{q}=3Kq+Kq^{2}+K.$
\item \label{enu:word map rewritten part 3}If $w$ is not the identity
or a proper power then, for all $i<2K,$ 
\[
\phi_{w}^{(i)}(0)=0.
\]
For proper powers, this holds for all $i<K$. 
\end{enumerate}
\end{thm}

We collect some facts about the polynomial 
\[
g_{L,K}(x)=\prod_{c=1}^{KL}(1-cx)^{L}\left[\prod_{j=1}^{2K}\left(1-(j-1)x\right)\right]^{L}
\]
in the following lemma.
\begin{lem}
\label{lem:Facts about q polynomial}Fix an integer $L>0.$ Then,
for every $t$ satisfying 
\[
0\leq t\leq\frac{1}{2}\frac{1}{KL^{2}(KL+1)+L(2K-1)(2K)}
\]
 and every integer $i\geq0:$
\begin{enumerate}
\item \label{enu:facts about q part 1}$\frac{1}{2}\leq g_{L,K}(t)\leq1$
\item \label{enu:facts about q part 2}$|g_{L,K}^{(i)}(t)|\leq\left(K^{2}L^{3}+L(2K-1)^{2}\right)^{i}$
\item \label{enu:facts about q part 3}$\left|\left(\frac{1}{g_{L,K}}\right)^{(i)}(t)\right|\leq(2i)!!2^{i+1}\left(K^{2}L^{3}+L(2K-1)^{2}\right)^{i}.$
\end{enumerate}
\end{lem}

\begin{proof}
Fix any $0\leq t\leq\frac{1}{2}\frac{1}{KL^{2}(KL+1)+L(2K-1)(2K)}.$
That $g_{L,K}(t)\leq1$ follows immediately from the fact that 
\[
|t|<\min\Big\{\frac{1}{KL},\frac{1}{2K-1}\Big\}.
\]
 To see that $\frac{1}{2}\leq g_{L,K}(t),$ observe that, in the regime
$0<z<\frac{1}{2},$ we have $(1-z)\geq e^{-2z}.$ Since 
\[
t<\min\Big\{\frac{1}{2KL},\frac{1}{2(2K-1)}\Big\},
\]
we have 
\[
\begin{aligned} & g_{L,K}(t)\\
\geq & \exp\left(-2\left(\sum_{c=1}^{KL}Lct+\sum_{j=1}^{2K}L(j-1)t\right)\right)\\
= & \exp\left(-2Lt\left(1+\dots+KL+1+\dots+2K-1\right)\right)\\
= & \exp\left(-2Lt\left(\frac{KL(KL+1)}{2}+\frac{(2K-1)2K}{2}\right)\right)\\
= & \exp\left(-t\left(KL^{2}(KL+1)+L(2K-1)(2K)\right)\right)\\
\geq & \frac{1}{2},
\end{aligned}
\]
where the final inequality follows from the fact that 
\[
t\leq\frac{1}{2}\frac{1}{KL^{2}(KL+1)+L(2K-1)(2K)}<\frac{-\log\left(\frac{1}{2}\right)}{KL^{2}(KL+1)+L(2K-1)(2K)}.
\]
For Part \ref{enu:facts about q part 2}, we begin by rewriting 
\[
g_{L,K}(t)=\prod_{z=1}^{KL^{2}+2KL}\left(1-b_{z}t\right)
\]
and then differentiate using the Leibniz rule. The $i^{\mathrm{th}}$
derivative is a sum of terms in which $i$ terms are derived once
(so are equal to $-b_{z}$) and the others are not derived. Each factor
that is not derived is of the form $1-b_{z}t,$ which belongs to $(0,1].$
So we can bound the derivative using only the terms that are derived:
\[
\begin{aligned} & \left|g_{L,K}^{(i)}(t)\right|\\
\leq & \sum_{\underset{\mathrm{distinct}}{u_{1},\dots,u_{i}}}i!\prod_{\alpha=1}^{i}b_{u_{\alpha}}\\
\leq & \left(\sum_{z=1}^{KL^{2}+2KL}b_{z}\right)^{i}\\
= & \left(\sum_{c=1}^{KL}Lc+\sum_{j=1}^{2K}L(j-1)\right)^{i}\\
\leq & \left(K^{2}L^{3}+L(2K-1)^{2}\right)^{i}.
\end{aligned}
\]

To prove Part \ref{enu:facts about q part 3}, we use that the $i^{\mathrm{th}}$
derivative of $\frac{1}{g_{L,K}}$ can be written as a product of
$(2i)!!$ terms of the form 
\[
\frac{\pm g_{L,K}^{(\alpha_{1})}\dots g_{L,K}^{(\alpha_{i})}}{g_{L,K}^{i+1}},
\]
where $\alpha_{1},\dots,\alpha_{i}\in\mathbb{Z}_{\geq0}$ and $\alpha_{1}+\dots+\alpha_{i}=i.$
By Part \ref{enu:facts about q part 2}, each term is bounded by 
\[
\frac{\prod_{j}\left(K^{2}L^{3}+L(2K-1)^{2}\right)^{\alpha_{j}}}{g_{L,K}^{i+1}}
\]
and, by Part \ref{enu:facts about q part 1}, this is bounded by 
\[
2^{i+1}\prod_{j}\left(K^{2}L^{3}+L(2K-1)^{2}\right)^{\alpha_{j}}=2^{i+1}\left(K^{2}L^{3}+L(2K-1)^{2}\right)^{i},
\]
from which Part \ref{enu:facts about q part 3} follows. 
\end{proof}
Each $w\in F_{r}$ defines a map $w\mapsto\phi_{w}.$ We extend this
by linearity to define a map $x\mapsto\phi_{x}$ for each 
\[
x=\sum_{w\in F_{r}}x(w)w\in\mathbb{C}\left[F_{r}\right].
\]
Then we can prove the following.
\begin{lem}
\label{lem:sup of phi_x^(i)}For any $x\in\mathbb{C}_{\leq q}\left[F_{r}\right]$,
for any $i\leq2K$ and for $n\geq Kq$, 
\[
\sup_{t\in\left[0,\frac{1}{2D_{q}^{2}}\right]}\frac{\left|\phi_{x}^{(i)}(t)\right|}{i!}\leq h(i,q)\left\Vert x\right\Vert _{C^{*}\left(F_{r}\right)},
\]
where 
\[
h(i,q)=4\left(\frac{CD_{q}^{4}}{i^{2}}\right)^{i}.
\]
\end{lem}

\begin{proof}
Let $P=g_{q,K}\phi_{x}.$ Then $P$ is a polynomial of bounded degree
$d\leq D_{q}$ and so we can bound its derivatives using Lemma \ref{lem: bounding sup pf kth derivative of polynomial }
if we can bound $P$ itself. For $n\geq Kq,$ we have $0<g_{q,K}\left(\frac{1}{n}\right)<1,$
so that 
\[
\begin{aligned} & \left|P\left(\frac{1}{n}\right)\right|\\
\leq & \left|\phi_{x}\left(\frac{1}{n}\right)\right|\\
= & \frac{1}{n^{K}}\left|\mathbb{E}_{x}\left[\mathrm{Tr_{\Sigma_{n,K}}}\right]\right|\\
\underset{(\dagger)}{\leq} & \frac{1}{n^{K}}\dim\left(\Sigma_{n,K}\right)\left\Vert x\right\Vert _{C^{*}\left(F_{r}\right)}\\
\underset{(\ref{eq:dimesnion of SIGMA})}{\leq} & \left\Vert x\right\Vert _{C^{*}\left(F_{r}\right)}.
\end{aligned}
\]
The inequality $(\dagger)$ follows since the map $w\mapsto\Sigma_{n,K}\left(w\left(\sigma_{1},\dots,\sigma_{r}\right)\right)$
is a unitary representation of $F_{r}$ for every $\left(\sigma_{1},\dots,\sigma_{r}\right)\in S_{n}^{r}.$
So we have 
\[
\left\Vert \Sigma_{n,K}\left(x\left(\sigma_{1},\dots,\sigma_{r}\right)\right)\right\Vert \leq\left\Vert x\right\Vert _{C^{*}\left(F_{r}\right)}
\]
almost surely, and since we can bound the trace of a matrix by its
norm multiplied by its size, we can bound 
\[
\left|\mathbb{E}_{x}\left[\mathrm{Tr_{\Sigma_{n,K}}}\right]\right|\leq\dim\left(\Sigma_{n,K}\right)\left\Vert x\right\Vert _{C^{*}\left(F_{r}\right)}.
\]
By Lemma \ref{lem: bounding sup pf kth derivative of polynomial },
for any integer $j,$ we can bound 
\[
\sup_{t\in\left[0,\frac{1}{2D_{q}^{2}}\right]}\frac{\left|P^{(j)}(t)\right|}{j!}\le\frac{2^{2j+1}D_{q}^{4j}}{j!(2j-1)!!}\left\Vert x\right\Vert _{C^{*}\left(F_{r}\right)}=\frac{2^{3j+1}D_{q}^{4j}}{(2j)!}\left\Vert x\right\Vert _{C^{*}\left(F_{r}\right)},
\]
which, using Stirling's formula, is bounded above by 
\begin{equation}
2\left(\frac{CD_{q}^{4}}{j^{2}}\right)^{j}\left\Vert x\right\Vert _{C^{*}\left(F_{r}\right)}.\label{eq: bound P(i) using ||x||}
\end{equation}
We then have 
\[
\frac{\phi_{x}^{(i)}}{i!}=\frac{1}{i!}\sum_{j=0}^{\min(i,D_{q})}\begin{pmatrix}i\\
j
\end{pmatrix}P^{(j)}\left(\frac{1}{g_{q,K}}\right)^{(i-j)}=\sum_{j=0}^{\min(i,D_{q})}\frac{P^{(j)}}{j!}\frac{1}{(i-j)!}\left(\frac{1}{g_{q,K}}\right)^{(i-j)}.
\]
Note that $i\leq2K\leq D_{q}=Kq^{2}+3Kq+K,$ so that $\min(i,D_{q})=i.$
By Lemma \ref{lem:Facts about q polynomial} Part \ref{enu:facts about q part 3}
, for $t\in\left[0,\frac{1}{2D_{q}^{2}}\right]$, 
\[
\frac{1}{(i-j)!}\left|\left(\frac{1}{g_{q,K}}\right)^{(i-j)}(t)\right|\le2^{2(i-j)+1}\left(K^{2}q^{3}+q(2K-1)^{2}\right)^{i-j}.
\]
Combining this with (\ref{eq: bound P(i) using ||x||}) we obtain
that, in the range $t\in\left[0,\frac{1}{2D_{q}^{2}}\right],$
\[
\begin{aligned}\left|\frac{\phi_{x}^{(i)}}{i!}\right|\le & \ 4\left\Vert x\right\Vert _{C^{*}\left(F_{r}\right)}\sum_{j=0}^{i}\left(\frac{CD_{q}^{4}}{j^{2}}\right)^{j}4^{(i-j)}\left(K^{2}q^{3}+q(2K-1)^{2}\right)^{i-j}\\
\le & \ 4\left(CD_{q}^{2}\right)^{i}\left\Vert x\right\Vert _{C^{*}\left(F_{r}\right)}\sum_{j=0}^{i}\left(\frac{D_{q}^{2}}{j^{2}}\right)^{j},
\end{aligned}
\]
where the final inequality follows from the fact that $K^{2}q^{3}+q(2K-1)^{2}\le D_{q}^{2}.$
Then observe that, for each $j$, 
\[
\frac{\left(\frac{D_{q}^{2}}{j^{2}}\right)^{j}}{\left(\frac{D_{q}^{2}}{(j+1)^{2}}\right)^{j+1}}=\left(\frac{j+1}{j}\right)^{2j}\left(\frac{j+1}{D_{q}}\right)^{2}\le e^{2},
\]
so that $\sum_{j=0}^{i}\left(\frac{D_{q}^{2}}{j^{2}}\right)^{j}\le(1+e^{2})^{i}\left(\frac{D_{q}^{2}}{i^{2}}\right)^{i},$
which proves the lemma. 
\end{proof}
For each integer $i\geq0,$ we define a map 
\[
\psi_{i}:x\in\C\left[F_{r}\right]\mapsto\frac{\phi_{x}^{(K+i)}(0)}{(K+i)!}\in\C.
\]
We want to show that $\psi_{i}$ is tempered for each $i<K$ and that
it satisfies the polynomial bound property in Proposition \ref{prop: criterion for strong convergence}. 
\begin{lem}
\label{lem: psi extends to bounded linear map}For every integer $i$
with $0\leq i\leq K$, there is a polynomial $P$ of degree $4K+4i+1$
such that $\left|\psi_{i}(x)\right|\leq P_{n}(q)\left\Vert x\right\Vert _{C^{*}\left(F_{r}\right)}$
for every $q$ and for every $x\in\C_{\leq q}\left[F_{r}\right].$
\end{lem}

\begin{proof}
By Lemma \ref{lem:sup of phi_x^(i)}, for each $q$ and for each $x\in\C_{\leq q}\left[F_{r}\right],$
we have 
\begin{equation}
\left|\frac{\phi_{x}^{(K+i)}(0)}{(K+i)!}\right|\leq h(K+i,q)\left\Vert x\right\Vert _{C^{*}\left(F_{r}\right)}.\label{eq: bound on |phi| k+1th derivative BEFORE Stirling}
\end{equation}
Moreover, 
\[
\sup_{q\geq1}\frac{h(K+i,q)}{(1+q^{2})^{4K+4i+1}}<\infty,
\]
from which the lemma follows. 
\end{proof}
\begin{lem}
\label{lem: psi(f( =00005Cmu)) is equal to zero}For any $i<K,$ the
function $\psi_{i}$ is tempered.
\end{lem}

\begin{proof}
We will use Proposition \ref{prop: criterion for being tempered}
by showing that, for any reasonable probability measure $\mu$ with
associated random walk $\left(g_{n}\right)_{n},$ 
\[
\limsup_{n}\left(\mathbb{E}|\psi_{i}(g_{n})|\right)^{\frac{1}{n}}\leq\rho(\mu).
\]
So, let $\mu$ be a reasonable probability measure, $(g_{n})_{n}$
the associated random walk on $F_{r}.$ If $\mu$ is supported in
$\C_{\leq q}\left[F_{r}\right],$ then $g_{n}\in\C_{\leq qn}\left[F_{r}\right].$
Then we have 
\[
\begin{aligned} & \mathbb{E}|\psi_{i}(g_{n})|\\
\leq & C_{i}(1+q^{2}n^{2})^{4K+4i+1}\mathbb{P}\left(\psi_{i}(g_{n})\neq0\right)\\
= & C_{i}(1+q^{2}n^{2})^{4K+4i+1}\mathbb{P}\left(g_{n}\ \mathrm{is\ a\ proper\ power}\right)\\
\leq & C_{i}C_{\mu}(1+q^{2}n^{2})^{4K+4i+1}n^{5}\rho(\mu)^{n}.
\end{aligned}
\]
The first inequality follows from Lemma \ref{lem: psi extends to bounded linear map}.
Indeed, $g_{n}\in\C_{\leq qn}\left[F_{r}\right],$ so there is some
constant $C_{i}$ for which $|\psi(g_{n})|\leq C_{i}\left(1+q^{2}n^{2}\right)^{4K+4i+1}.$
The final inequality follows from Proposition \ref{prop: bound on prob that random walk ends on power}.
It follows that 
\[
\limsup_{n}\left(\mathbb{E}|\psi_{i}(g_{n})|\right)^{\frac{1}{n}}\leq\rho(\mu),
\]
so by Proposition \ref{prop: criterion for being tempered}, $\psi_{i}$
is tempered. 
\end{proof}
The next step is showing that the second condition in Proposition
\ref{prop: criterion for strong convergence} holds. We will need
the following lemma. 
\begin{lem}
\label{lem: bound on sup h(2k,x)exp(-q)}For any $\epsilon>0,$ we
have $\sup_{q\geq1}h(2K,q)\exp\left(-\frac{q}{\log(2+q)^{2}}\right)\leq\left(C^{2}K^{20+\epsilon}\right)^{K}.$
\end{lem}

\begin{proof}
Fix $\epsilon>0.$ Given any $b>1,$ there exists $a>0$ such that
\[
\log\left(2+q\right)^{2}\leq q^{1/b}
\]
 for all $q>a$. 

So, for $q>a,$ we have 
\begin{equation}
4\left(CD_{q}^{4}\right)^{2K}\exp\left(\frac{-q}{\log\left(2+q\right)^{2}}\right)\leq4\left(C\left(Kq^{2}+3Kq+K\right)^{4}\right)^{2K}\exp\left(-q^{\frac{b-1}{b}}\right)\label{eq: bounding h(2k,q)exp(frac) with h(2k,q)exp(-q)}
\end{equation}
which is bounded above by 
\begin{equation}
\left(C^{2}K^{20+\epsilon}\right)^{K}\label{eq: bound for sup_q ...}
\end{equation}
for sufficiently large $b$.\footnote{To see this, differentiate the RHS of (\ref{eq: bounding h(2k,q)exp(frac) with h(2k,q)exp(-q)})
to see that the maximum is obtained around $q=CK^{\frac{b}{b-1}}$and
substitute this into $h(2K,q)$.} 

For $q\leq a$, we obtain the bound 
\[
\left(C'K^{2}\right)^{2K}
\]
by evaluating $h(2K,a)$ and ignoring the exponential term (since
it is $\leq1$). This is less than (\ref{eq: bound for sup_q ...})
(when $C$ is large enough) from which the lemma follows. 
\end{proof}
We will also need the following observation to be used in Lemma \ref{lem: trace SIGMA - tau - taylor series},
as well as our proof of Theorem \ref{thm:Main theorem strong convergence}
-- the map 
\[
u:x\in\C\left[F_{r}\right]\mapsto\tau(x)\frac{\phi_{e}^{(K)}(0)}{K!}\in\C
\]
is tempered. Indeed, if we write $\dim\Sigma_{n,K}=a_{0}+a_{1}n+\dots+a_{k}n^{K},$
one sees that $\frac{\phi_{e}^{(K)}(0)}{K!}=a_{0},$ a constant (in
fact, this constant is always $(-1)^{K}$), and our observation follows
from the fact that $x\mapsto\tau(x)$ is tempered. Moreover, this
map obviously satisfies the polynomial bound $|u(x)|\le\|x\|_{C^{*}\left(F_{r}\right)}.$
\begin{lem}
\label{lem: trace SIGMA - tau - taylor series}Let $w(q)=\exp\left(\frac{q}{\log\left(2+q\right)^{2}}\right).$
Then, for every $q\geq1,$ for every $n\geq Kq$ and for every $x\in\C_{\leq q}\left[F_{r}\right],$
and any $\epsilon>0,$
\begin{equation}
\left|\mathbb{E}\left[\mathrm{Tr}\big(\Sigma_{n,K}\left(x(\sigma_{1},\dots,\sigma_{r})\right)-\tau(x)\mathrm{Id}\big)\right]+\tau(x)\frac{\phi_{e}^{(K)}(0)}{K!}-\sum_{i=0}^{K-1}\frac{\psi_{i}(x)}{n^{i}}\right|\label{eq: ecpected trace minus sum is bounded}
\end{equation}
is bounded by 
\[
\frac{\left(C^{2}K^{20+\epsilon}\right)^{K}}{n^{K}}w(q)\left\Vert x\right\Vert _{C^{*}\left(F_{r}\right)}.
\]
\end{lem}

\begin{proof}
We have 
\[
\mathbb{E}\left[\mathrm{Tr}\left(\tau(x)\mathrm{Id}\right)\right]=\tau(x)\dim\Sigma_{n,K}=\tau(x)n^{K}\sum_{i=0}^{K-1}\frac{\phi_{e}^{(i)}(0)}{i!n^{i}}+\tau(x)\frac{\phi_{e}^{(K)}(0)}{K!}
\]
and, since $n\geq Kq,$ by Part \ref{enu:word map rewritten part 3}
of Theorem \ref{thm: word map thm rewritten}, we have 
\[
\sum_{i=0}^{K-1}\frac{\psi_{i}(x)}{n^{i}}=n^{K}\sum_{i=K}^{2K-1}\frac{\phi_{x}^{(i)}(0)}{i!n^{i}}=n^{K}\left[\sum_{w\neq e}x(w)\sum_{i=0}^{2K-1}\frac{\phi_{w}^{(i)}(0)}{i!n^{i}}+\tau(x)\sum_{i=K}^{2K-1}\frac{\phi_{e}^{(i)}(0)}{i!n^{i}}\right].
\]
Combining these observations with Part \ref{enu: word map rewritten part 1}
of Theorem \ref{thm: word map thm rewritten}, we see that the LHS
of (\ref{eq: ecpected trace minus sum is bounded}) is equal to 
\[
n^{K}\left|\phi_{x}\left(\frac{1}{n}\right)-\sum_{i=0}^{2K-1}\frac{\phi_{x}^{(i)}(0)}{i!n^{i}}\right|.
\]
By Taylor's inequality, this is less than or equal to 
\[
\frac{n^{K}}{n^{2K}(2K)!}\left|\phi_{x}^{(2K)}\left(\frac{1}{n}\right)\right|.
\]
If we further assume that $n\geq2D_{q}^{2},$ then by Lemma \ref{lem:sup of phi_x^(i)},
this is bounded by 
\begin{equation}
\frac{h(2K,q)}{n^{K}}\left\Vert x\right\Vert _{C^{*}\left(F_{r}\right)}.\label{eq: bound for expected trace minus sum}
\end{equation}

\noindent If $n\leq2D_{q}^{2}$, then we can still bound the left
hand side of (\ref{eq: ecpected trace minus sum is bounded}) by 
\[
\left(2\dim\left(\Sigma_{n,K}\right)+1\right)\left\Vert x\right\Vert _{C^{*}\left(F_{r}\right)}+\sum_{i=0}^{K-1}\frac{\left|\psi_{i}(x)\right|}{n^{i}}
\]
by using the triangle inequality. By (\ref{eq:dimesnion of SIGMA})
and Lemma \ref{lem:sup of phi_x^(i)}, this is bounded above by 
\[
n^{K}\left[h(0,q)+\sum_{i=0}^{K-1}\frac{1}{n^{K+i}}h(K+i,q)\right]\left\Vert x\right\Vert _{C^{*}\left(F_{r}\right)}.
\]
This is less than (\ref{eq: bound for expected trace minus sum})
whenever the constant $C$ is large enough. 

By Lemma \ref{lem: bound on sup h(2k,x)exp(-q)}, we have
\[
\sup_{q\geq1}h(2K,q)\exp\left(-\frac{q}{\log(2+q)^{2}}\right)\leq\left(C^{2}K^{20+\epsilon}\right)^{K}
\]
and the lemma follows. 
\end{proof}
\begin{proof}[Proof of Theorem \ref{thm:Main theorem strong convergence}]
\label{thm: SIGMA converges to reduced c*} Fix $\alpha<\frac{1}{20},$
say $\alpha=\frac{1}{20}-\epsilon',$ with $0<\epsilon'\leq\frac{1}{20}.$
For each $n,$ and for any $K\leq n^{\alpha}$, let $\Pi_{n,K}$ be
a random representation of $F_{r}$ given by 
\[
\Pi_{n,K}(w)=\Sigma_{n,K}\left(w\left(\sigma_{1},\dots,\sigma_{r}\right)\right).
\]
Then, by Lemma \ref{lem: trace SIGMA - tau - taylor series}, for
every $q$ and every $x\in\C_{\leq q}\left[F_{r}\right],$ we have
\[
\left|\mathbb{E}\mathrm{Tr}\left(\Pi_{n,K}(x)\right)-T_{n}(x)\right|\leq\epsilon_{n}w(q)\left\Vert x\right\Vert _{C^{*}\left(F_{r}\right)},
\]
where, for some $\epsilon$ satisfying $0<\epsilon<\frac{20\epsilon'}{\frac{1}{20}-\epsilon'}$,
we have
\[
\epsilon_{n}=\frac{\left(C^{2}K^{20+\epsilon}\right)^{K}}{n^{K}}\leq\left(\frac{C^{2}}{n^{1-(20+\epsilon)\alpha}}\right)^{K}
\]
and 
\[
T_{n}(x)=\dim\left(\Sigma_{n,K}\right)\tau(x)-\tau(x)\frac{\phi_{e}^{(K)}(0)}{K!}+\sum_{i=0}^{K-1}\frac{\psi_{i}(x)}{n^{i}}.
\]

By Lemma \ref{lem: psi(f( =00005Cmu)) is equal to zero}, $T_{n}$
is tempered, since it is a finite sum of tempered functions. Additionally,
$T_{n}$ satisfies the polynomial bound
\[
\left|T_{n}(x)\right|\leq P_{n}(q)\left\Vert x\right\Vert _{C^{*}\left(F_{r}\right)}.
\]
So by Proposition \ref{prop: criterion for strong convergence}, with
$T_{n}$ in place of $u_{n}$, $\forall\delta>0$, and any $z\in\C\left[F_{r}\right],$
\[
\mathbb{P}\left[\left\Vert \Pi_{n,K}(z)\right\Vert >\left\Vert z\right\Vert _{C_{\lambda}^{*}(F_{r})}+\delta\right]\leq C(z,\delta)\left(\frac{C^{2}}{n^{1-(20+\epsilon)\alpha}}\right)^{K}.
\]
For $k_{n}\leq n^{\alpha},$ we have 
\begin{equation}
\begin{aligned} & \mathbb{P}\left[\left\Vert \pi_{n,k_{n}}(z)\right\Vert >\left\Vert z\right\Vert _{C_{\lambda}^{*}(F_{r})}+\delta\right]\\
\leq & \sum_{1\leq K\leq n^{\alpha}}\mathbb{P}\left[\left\Vert \Pi_{n,K}(z)\right\Vert >\left\Vert z\right\Vert _{C_{\lambda}^{*}(F_{r})}+\delta\right].
\end{aligned}
\label{eq: =00005CrhoisleqthanSigma}
\end{equation}
Write $A=1-(20+\epsilon)\alpha,$ which is $>0$ by our choice of
$\epsilon.$ Then the RHS of (\ref{eq: =00005CrhoisleqthanSigma})
is equal to 
\[
C(z,\delta)\left(\frac{C^{2}}{n^{A}}\right)\left(\frac{1-\left(\frac{C^{2}}{n^{A}}\right)^{n^{\alpha}-1}}{1-\left(\frac{C^{2}}{n^{A}}\right)}\right),
\]
which $\to0$ as $n\to\infty.$ So, for any $k_{n}\leq n^{\alpha},$
\begin{equation}
\mathbb{P}\left[\left\Vert \pi_{n,k_{n}}(z)\right\Vert >\left\Vert z\right\Vert _{C_{\lambda}^{*}(F_{r})}+\delta\right]\overset{n\to\infty}{\longrightarrow}0.\label{eq:lower bound}
\end{equation}
By \cite[Lemma 5.14]{MAgeeDelaSalle}, there exist $y_{1},\dots,y_{m}\in\C\left[F_{r}\right]$
and $\delta'=\delta'(\delta),$ such that if, for every $i,$ 
\begin{equation}
\left\Vert \pi_{n,k_{n}}\left(y_{i}\right)\right\Vert \leq\left\Vert y_{i}\right\Vert _{C_{\lambda}^{*}(F_{r})}+\min(\delta,\delta'),\label{eq: rho (yi) is leq ||y_i||+delta}
\end{equation}
then 
\begin{equation}
\left\Vert \pi_{n,k_{n}}\left(z\right)\right\Vert \geq\left\Vert z\right\Vert _{C_{\lambda}^{*}(F_{r})}-\delta.\label{eq: rho(x) is geq ||x||-delta}
\end{equation}
By (\ref{eq:lower bound}), we know that (\ref{eq: rho (yi) is leq ||y_i||+delta})
holds with probability $\to1$ as $n\to\infty,$ so then (\ref{eq: rho(x) is geq ||x||-delta})
also holds with probability $\to1$ as $n\to\infty.$ Therefore, 
\[
\mathbb{P}\left[\sup_{k_{n}\leq n^{\alpha}}\left|\left\Vert \pi_{n,k_{n}}\left(z\right)\right\Vert -\left\Vert z\right\Vert _{C_{\lambda}^{*}(F_{r})}\right|>\delta\right]\overset{n\to\infty}{\longrightarrow}0.
\]
\end{proof}

\section{Proof of Theorem \ref{thm: word map main theorem}\label{sec: proof of word map theorem}}

\subsection{Overview of proof}

Rather than working directly with the character $\chi^{\lambda^{+}(n)},$
we use Corollary \ref{cor: bitrace corollary} so that instead our
task is to compute the trace of $\Q_{\lambda,n}\circ w\left(g_{1},\dots,g_{r}\right)$
on $\cnk.$ 

This amounts to counting the number of standard basis vectors of $\cnk$
fixed by this map. Recall that $\Q_{\lambda,n}^{2}=\Q_{\lambda,n}$
and that, since it is a linear combination of $p_{\pi},$ the action
of $\Q_{\lambda,n}$ commutes with the action of any $g\in S_{n}$
on $\cnk.$ With this in mind, if $w=x_{1}x_{2}x_{1}^{-1}x_{2}^{-1}\in F_{2}$
for example, what we are actually interested in is 
\[
\mathrm{tr}_{\cnk}\left(g_{1}\circ\Q_{\lambda,n}\circ g_{2}\circ\Q_{\lambda,n}\circ g_{1}^{-1}\circ\Q_{\lambda,n}\circ g_{2}^{-1}\circ\Q_{\lambda,n}\right),
\]
where each $g_{i}\in S_{n}$. 

We compute the expected trace by using the Weingarten calculus for
$S_{n}.$ A key component is our refinement of the usual Weingarten
calculus for $S_{n},$ which uses the fact that we are operating within
$\A$ to show that the contribution to the trace from all but a specific
family of partitions is zero. 

We obtain a combinatorial formula for the trace in \S\ref{subsec:Combinatorial-integration}.
In \S\ref{subsec:Graphical-Interpretation}, we construct graphs
from the combinatorial data formula for the expected character. From
these graphs we construct new graphs in which the asymptotic bound
for our trace formula is encoded by the Euler characteristic. 

In \S\ref{subsec:Obtaining-the-bound}, we prove a variant of a theorem
of Louder and Wilton that relates the Euler characteristic of a graph
with the number of immersed $w$--cycles (see \cite[Theorem 1.2]{LouderWilton})
to obtain our final bound for the expected character. 

\subsection{Expected character as a ratio of polynomials in $\frac{1}{n}$\label{subsec:polynomial ratio character}}

In addition to our main theorem on word maps, we give another formulation
for the expected stable irreducible character of a $w$--random permutation
as the ratio of two polynomials in $\frac{1}{n}.$ This formulation
is required for the methods in \S\ref{sec: Proof of strong convergence}.

For any $L,k\in\mathbb{Z}_{>0},$ we define
\[
g_{L,k}(x)\overset{\mathrm{def}}{=}\prod_{c=1}^{L}(1-cx)^{L}\left[\prod_{j=1}^{2k}\left(1-(j-1)x\right)\right]^{L}.
\]

\begin{thm}
\label{thm: word map polynomial theorem}Let $w\in F_{r}=\left\langle x_{1},\dots,x_{r}\right\rangle $
be a word in the free group with $r$ generators, $w\neq e$. Let
$k\in\mathbb{Z}_{>0}$ and $\lambda\vdash k$. If $l(w)\leq q,$ there
is a polynomial $P_{w,\lambda,q}\in\mathbb{Q}\left[x\right]$ such
that, for $n\geq l(w)k,$
\[
\mathbb{E}_{w}\left[\chi^{\lambda^{+}(n)}\right]=\frac{P_{w,\lambda,q}\left(\frac{1}{n}\right)}{g_{q,k}\left(\frac{1}{n}\right)},
\]
with $\deg\left(P_{w,\lambda,q}\right)\leq3kq+q^{2}.$ 
\end{thm}

The proof of Theorem \ref{thm: word map polynomial theorem} is given
in \S\ref{subsec:ratio of polynomials proof} and it can be followed
from (\ref{eq: expected character with pi, weingarten and N}).

\subsection{Combinatorial integration\label{subsec:Combinatorial-integration}}

Fix a word $w\in\f=\left\langle x_{1},\dots,x_{r}\right\rangle .$
We will assume that $w$ is not the identity and is not primitive
and that $w$ is cyclically reduced. We will also assume that every
$x_{i}$ appears at least twice in $w$. If $x_{j}$ did not appear
in $w,$ then we could consider $w\in F_{r-1}=\left\langle x_{1},\dots,x_{j-1},x_{j+1},\dots,x_{r}\right\rangle $
and then proceed. If any $x_{j}$ appeared only once, then $w$ would
be primitive and so $\mathbb{E}_{w}\left[\chi^{\lambda^{+}(n)}\right]=0.$ 

We will write 
\[
w=f_{1}^{\epsilon_{1}}f_{2}^{\epsilon_{2}}\dots f_{l(w)}^{\epsilon(w)},
\]
where each $f_{i}\in\{x_{1},\dots,x_{r}\}$ and each $\epsilon_{i}\in\{1,-1\}.$
We will write $|w|_{x_{i}}$ for the number of $j$ such that $f_{j}=x_{i}.$

Suppose that $\{v_{p}\}$ is an orthonormal basis for $\mathcal{U}_{\lambda^{+}(n)}\cong V^{\lambda^{+}(n)}\otimes V^{\lambda}.$
Then, for any $\left(g_{x_{1}},\dots,g_{x_{r}}\right)\in S_{n}^{r},$
we have 
\[
\chi^{\lambda^{+}(n)}\left(w\left(g_{x_{1}},\dots,g_{x_{r}}\right)\right)=\frac{1}{d_{\lambda}}\mathrm{tr}_{\mathcal{U}_{\lambda^{+}(n)}}\left(w\left(g_{x_{1}},\dots,g_{x_{r}}\right)\right),
\]
so that 
\begin{equation}
\begin{aligned} & \chi^{\lambda^{+}(n)}\left(w\left(g_{x_{1}},\dots,g_{x_{r}}\right)\right)\\
= & \frac{1}{d_{\lambda}}\sum_{p_{i}}\left\langle g_{f_{1}}^{\epsilon_{1}}v_{p_{2}},\ v_{p_{1}}\right\rangle \left\langle g_{f_{2}}^{\epsilon_{2}}v_{p_{3}},\ v_{p_{2}}\right\rangle \dots\left\langle g_{f_{l(w)}}^{\epsilon_{l(w)}}v_{p_{1}},\ v_{p_{l(w)}}\right\rangle .
\end{aligned}
\label{eq: character as sum of inner products}
\end{equation}
For each $p$, write $v_{p}$ in the standard orthonormal basis of
$\cnk$: 
\[
v_{p}=\sum_{I}\beta_{p,I}e_{I}.
\]
Recall that $\mathcal{U}_{\lambda^{+}(n)}\subset D_{k}(n),$ so that
we may assume that the above sum is over all $I=\left(i_{1},\dots,i_{k}\right)$
\emph{with all indices distinct. }This is an important observation. 

We introduce some new notation to avoid the cumbersome general expression
for $\mathop{\mathbb{E}}_{w}\left[\chi^{\lambda^{+}(n)}\right].$
For each $f\in\{x_{1},\dots,x_{r}\},$ we will write 
\[
\sum_{I^{f}}
\]
in place of 
\[
\sum_{I_{f}^{1},\dots I_{f}^{|w|_{f}}}.
\]
For any pair $I,J$ of multi--indices and $\epsilon\in\{1,-1\},$
we define 
\[
(I/J)(\epsilon)=\begin{cases}
J & \mathrm{if}\ \epsilon=1\\
I & \mathrm{if}\ \epsilon=-1.
\end{cases}
\]

In our expression, we will simply write $(I/J)$ in place of $(I/J)(\epsilon)$
where it is clear which epsilon we are considering. To be even more
clear, in the sum below, for each $f_{i},$ if $\epsilon_{i}=1,$
then the corresponding inner product has $e_{I_{f_{i}}^{z}}$ on the
LHS, where $z\in\{1,\dots,|w|_{f_{i}}\}$ denotes the number of times
$f_{i}$ has appeared in the subword $f_{1}^{\epsilon_{1}}\dots f_{i}^{\epsilon_{i}}.$
The corresponding pair of $\beta$--terms has $\beta_{p,I},$ with
$\beta_{p,J}$ conjugated. If $\epsilon_{i}=-1,$ we swap the positions
of $I$ and $J$.\footnote{We use this notation so that, in the graph construction detailed in
Section \ref{subsec:Graphical-Interpretation}, $\left(J_{f}^{i}\right)_{j}$
always represents the \emph{initial }vertex of an $f$--edge and
$\left(I_{f}^{i}\right)_{j}$ always represents the \emph{terminal
}vertex.}

We will write 
\[
\begin{aligned} & \prod_{f,w}\beta\\
= & \left(\beta_{p_{2},\left(I_{f_{1}}^{1}/J_{f_{1}}^{1}\right)}\right)\left(\bar{\beta}_{p_{1},\left(J_{f_{1}}^{1}/I_{f_{1}}^{1}\right)}\right)\dots\\
 & \dots\left(\beta_{p_{1},\left(I_{f_{l(w)}}^{|w|_{f_{l(w)}}}/J_{f_{l(w)}}^{|w|_{f_{l(w)}}}\right)}\right)\left(\bar{\beta}_{p_{l(w)},\left(J_{f_{l(w)}}^{|w|_{f_{l(w)}}}/I_{f_{l(w)}}^{|w|_{f_{l(w)}}}\right)}\right).
\end{aligned}
\]
With this notation, (\ref{eq: character as sum of inner products})
is equal to
\begin{equation}
\begin{aligned} & \frac{1}{d_{\lambda}}\sum_{p_{i}}\sum_{I^{f},J^{f}}\left(\prod_{f,w}\beta\right)\left\langle g_{f_{1}}^{\epsilon_{1}}e_{\left(I_{f_{1}}^{1}/J_{f_{1}}^{1}\right)},\ e_{\left(J_{f_{1}}^{1}/I_{f_{1}}^{1}\right)}\right\rangle \dots\\
 & \dots\left\langle g_{f_{l(w)}}^{\epsilon_{l(w)}}e_{\left(I_{f_{l(w)}}^{|w|_{f_{l(w)}}}/J_{f_{l(w)}}^{|w|_{f_{l(w)}}}\right)},e_{\left(J_{f_{l(w)}}^{|w|_{f_{l(w)}}}/I_{f_{l(w)}}^{|w|_{f_{l(w)}}}\right)}\right\rangle .
\end{aligned}
\label{eq: character with beta and inner product terms}
\end{equation}

See the example below. \emph{Without losing generality, we will always
assume that $f_{1}=x_{1}$ and that $\epsilon_{1}=1$. }
\begin{example}
\begin{flushleft}
\label{exa: expression for expected trace with w=00003D =00005Ba,b=00005D}Suppose
$w=x_{1}x_{2}x_{1}^{-1}x_{2}^{-1}.$ Then $\chi^{\lambda^{+}(n)}\left(w\left(g_{x_{1}},g_{x_{2}}\right)\right)$
is equal to
\[
\begin{aligned}\frac{1}{d_{\lambda}}\sum_{p_{i}}\sum_{\begin{aligned}I_{x_{1}}^{1},I_{x_{1}}^{2},I_{x_{2}}^{1},I_{x_{2}}^{2},\\
J_{x_{1}}^{1},J_{x_{1}}^{2},J_{x_{2}}^{1},J_{x_{2}}^{2}
\end{aligned}
}\left(\beta_{p_{2},I_{x_{1}}^{1}}\right)\left(\bar{\beta}_{p_{1},J_{x_{1}}^{1}}\right)\left(\beta_{p_{3},I_{x_{2}}^{1}}\right)\left(\bar{\beta}_{p_{2},J_{x_{2}}^{1}}\right)\left(\beta_{p_{4},J_{x_{1}}^{2}}\right)\left(\bar{\beta}_{p_{3},I_{x_{1}}^{2}}\right)\\
\left(\beta_{p_{1},J_{x_{2}}^{2}}\right)\left(\bar{\beta}_{p_{4},I_{x_{2}}^{2}}\right)\left\langle g_{x_{1}}e_{J_{x_{1}}^{1}},e_{I_{x_{1}}^{1}}\right\rangle \left\langle g_{x_{2}}e_{J_{x_{2}}^{1}},e_{I_{x_{2}}^{1}}\right\rangle \left\langle g_{x_{1}}^{-1}e_{I_{x_{1}}^{2}},e_{J_{x_{1}}^{2}}\right\rangle \left\langle g_{x_{2}}^{-1}e_{I_{x_{2}}^{2}},e_{J_{x_{2}}^{2}}\right\rangle .
\end{aligned}
\]

\par\end{flushleft}
\end{example}

We rewrite the inner product terms as products of matrix coefficients
-- for example, 
\[
\left\langle g_{x_{1}}e_{J_{x_{1}}^{1}},e_{I_{x_{1}}^{1}}\right\rangle =\left(g_{x_{1}}\right)_{\left(I_{x_{1}}^{1}\right)_{1}\left(J_{x_{1}}^{1}\right)_{1}}.\dots\left(g_{x_{1}}\right)_{\left(I_{x_{1}}^{1}\right)_{k}\left(J_{x_{1}}^{1}\right)_{k}}.
\]
Rearranging and taking the expectation over $g_{x_{1}},\dots,g_{x_{r}}\in S_{n},$
we obtain from (\ref{eq: character with beta and inner product terms})
that $\mathop{\mathbb{E}}_{w}\left[\chi^{\lambda^{+}(n)}\right]$
is equal to 
\begin{equation}
\frac{1}{d_{\lambda}}\sum_{p_{i}}\sum_{I^{f},J^{f}}\left(\prod_{f,w}\beta\right)\prod_{f\in\{x_{1},\dots,x_{r}\}}\int_{g_{I^{f},J^{f}}},\label{eq: expected character with beta and integral over matrix coefficient terms}
\end{equation}
where 
\[
\int_{g_{I^{f},J^{f}}}\overset{\mathrm{def}}{=}\int_{S_{n}}\prod_{i=1}^{|w|_{f}}\left(g_{f}\right)_{\left(I_{f}^{i}\right)_{1}\left(J_{f}^{i}\right)_{1}}.\dots\left(g_{f}\right)_{\left(I_{f}^{i}\right)_{k}\left(J_{f}^{i}\right)_{k}}dg_{f}.
\]
For each $f\in\{x_{1},\dots,x_{r}\}$ and for each fixed collection
of multi--indices $I_{f}^{1},\dots,I_{f}^{|w|_{f}},J_{f}^{1},\dots,J_{f}^{|w|_{f}}$,
this integral can be computed using the Weingarten calculus for the
symmetric group. Using (\ref{thm: weingarten calculus}) we have 
\[
\begin{aligned} & \int_{S_{n}}\prod_{i=1}^{|w|_{f}}\left(g_{f}\right)_{\left(I_{f}^{i}\right)_{1}\left(J_{f}^{i}\right)_{1}}.\dots\left(g_{f}\right)_{\left(I_{f}^{i}\right)_{k}\left(J_{f}^{i}\right)_{k}}dg_{f}\\
= & \sum_{\sigma_{f},\tau_{f}\in\mathrm{Part}\left(\left[|w|_{f}k\right]\right)}\delta_{\sigma_{f}}\left(J_{f}^{1}\sqcup J_{f}^{2}\sqcup\dots\sqcup J_{f}^{|w|_{f}}\right)\delta_{\tau_{f}}\left(I_{f}^{1}\sqcup I_{f}^{2}\sqcup\dots\sqcup I_{f}^{|w|_{f}}\right)\\
 & \mathrm{Wg}_{n,\left(|w|_{f}k\right)}\left(\sigma_{f},\tau_{f}\right).
\end{aligned}
\]
Now we give an improvement over the usual Weingarten calculus that
simplifies the above equation greatly. The benefit of the improvement
is that, instead of summing over \emph{all }set partitions $\sigma_{f},\tau_{f}\in\mathrm{Part}\left(\left[|w|_{f}k\right]\right),$
we show that we need only sum over set partitions $\sigma_{f},\tau_{f}$
that have a specific structure, since the contribution to (\ref{eq: expected character with beta and integral over matrix coefficient terms})
from the set partitions without this structure is zero.

To each partition we associate a diagram (in a similar way to \S\ref{subsec:Partition-algebra})
-- the diagram has $|w|_{f}$ rows, each containing exactly $k$
vertices. The vertices in row $R$ are labeled from $(R-1)k+1$ to
$Rk$, and two vertices are connected if and only if the corresponding
vertex labels are in the same block of of the partition. We will also
refer to the diagrams corresponding to $\sigma_{f}$ and $\tau_{f}$
as $\sigma_{f}$ and $\tau_{f}$ where this does not cause confusion.
\begin{lem}
\label{lem: sigma and tau have no two vertices in the same row connected}For
each $f\in\{x_{1},\dots,x_{r}\},$ if $\sigma_{f}$ has any two vertices
from the same row connected, then 
\[
\delta_{\sigma_{f}}\left(J_{f}^{1}\sqcup J_{f}^{2}\sqcup\dots\sqcup J_{f}^{|w|_{f}}\right)=0.
\]
The same is true for $\tau_{f}$ and $I_{f}^{1}\sqcup I_{f}^{2}\sqcup\dots\sqcup I_{f}^{|w|_{f}}.$ 
\end{lem}

\begin{proof}
Since $e_{J_{f}^{1}}\in D_{k}(n)=\dkn$, if any of the vertices in
the top row of $\sigma_{f}$ are in the same block, then 
\[
\delta_{\sigma_{f}}\left(J_{f}^{1}\sqcup J_{f}^{2}\sqcup\dots\sqcup J_{f}^{|w|_{f}}\right)=0.
\]
 Repeating this argument for $J_{f}^{2},\dots J_{f}^{|w|_{f}}$ proves
the claim for $\sigma_{f}$ and then repeating it for $I_{f}^{1}\sqcup I_{f}^{2}\sqcup\dots\sqcup I_{f}^{|w|_{f}}$
proves the claim for $\tau_{f}.$
\end{proof}
The next lemma asserts that every vertex of $\sigma_{f}$ and $\tau_{f}$
must be connected to at least one other vertex.
\begin{lem}
\label{lem: sigma and tau have k blocks of size |w|f}For any $f\in\{x_{1},\dots,x_{r}\},$
if $\sigma_{f}$ has any singletons, then 
\[
\begin{aligned} & \frac{1}{d_{\lambda}}\sum_{p_{i}}\sum_{I^{f},J^{f}}\left(\prod_{f,w}\beta\right)\\
 & \Bigg[\sum_{\tau_{f}\in\mathrm{Part}\left(\left[|w|_{f}k\right]\right)}\delta_{\sigma_{f}}\left(J_{f}^{1}\sqcup J_{f}^{2}\sqcup\dots\sqcup J_{f}^{|w|_{f}}\right)\delta_{\tau_{f}}\left(I_{f}^{1}\sqcup I_{f}^{2}\sqcup\dots\sqcup I_{f}^{|w|_{f}}\right)\\
 & \mathrm{Wg}_{n,\left(|w|_{f}k\right)}\left(\sigma_{f},\tau_{f}\right)\Bigg]=0.
\end{aligned}
\]
The same is true when we swap $\sigma_{f}$ and $\tau_{f}.$ 
\end{lem}

\begin{proof}
Without loss of generality, suppose that the vertex $q$ in the first
row of vertices of $\sigma_{f}$ is a singleton. For all variables
in the sum fixed except for $J_{f}^{1},J_{f}^{2},\dots,J_{f}^{|w|_{f}},$
assuming 
\[
\delta_{\tau_{f}}\left(I_{f}^{1}\sqcup I_{f}^{2}\sqcup\dots\sqcup I_{f}^{|w|_{f}}\right)\neq0,
\]
then the sum is equal to 
\[
\sum_{J_{f}^{1},\dots,J_{f}^{|w|_{f}}}\left(\bar{\beta}_{p_{1},J_{f}^{1}}\right)\dots\left(\beta_{p,J_{f}^{|w|_{f}}}\right)\delta_{\sigma_{f}}\left(J_{f}^{1}\sqcup J_{f}^{2}\sqcup\dots\sqcup J_{f}^{|w|_{f}}\right),
\]
 multiplied by some constant coming from the other (fixed) $\beta$--terms
and the Weingarten function $\mathrm{Wg}_{n,\left(|w|_{f}k\right)}\left(\sigma_{f},\tau_{f}\right).$ 

Then, for every fixed index $\left(J_{f}^{1}\right)_{1},\dots,\left(J_{f}^{1}\right)_{q-1},\left(J_{f}^{1}\right)_{q+1},\dots,\left(J_{f}^{1}\right)_{k}$
and every fixed multi--index $J_{f}^{2},\dots,J_{f}^{|w|_{f}}$ satisfying
$\sigma_{f},$ this is (some constant multiplied by) 
\[
\sum_{\left(J_{f}^{1}\right)_{q}=1}^{n}\bar{\beta}_{p_{1},J_{f}^{1}}.
\]
This is exactly the conjugate of the coefficient of $e_{(J_{f}^{1})_{1}}\otimes\dots\otimes e_{(J_{f}^{1})_{q-1}}\otimes e_{(J_{f}^{1})_{q+1}}\otimes\dots e_{(J_{f}^{1})_{k}}$
in $T_{q}\left(v_{p_{1}}\right).$ But then, $v_{p_{1}}$ belongs
to an orthonormal basis for $\u,$ and $\u\subset\A.$ In particular,
$v_{p}\in\ker\left(T_{q}\right),$ so that 
\[
\sum_{\left(J_{f}^{1}\right)_{q}=1}^{n}\bar{\beta}_{p_{1},J_{f}^{1}}=0.
\]
\end{proof}
\emph{Henceforth, we will denote by $\overset{\star}{\mathrm{Part}}\left(\left[|w|_{f}k\right]\right)$}
\emph{the set of set partitions of} $\left[|w|_{f}k\right]$ \emph{for
which the corresponding diagram has no singletons and has no two vertices
in the same row in the same connected component. }

\emph{We will write 
\[
\sum_{\sigma_{f},\tau_{f}}^{\star}
\]
 to indicate that the sum is over $\sigma_{f},\tau_{f}\in\refinedpart.$}

For any given $\sigma_{f}\in\refinedpart$ (and similarly for $\tau_{f}$),
and any collection of multi--indices $J_{f}^{1},\dots,J_{f}^{|w|_{f}},$
we will write 
\[
J_{f}^{\Sigma}\leftrightarrow\sigma_{f}
\]
to indicate that 
\[
\delta_{\sigma_{f}}\left(J_{f}^{1}\sqcup J_{f}^{2}\sqcup\dots\sqcup J_{f}^{|w|_{f}}\right)=1.
\]
Therefore, our new expression for $\mathop{\mathbb{E}}_{w}\left[\chi^{\lambda^{+}(n)}\right]$
is 
\begin{equation}
\begin{aligned}\frac{1}{d_{\lambda}}\sum_{\sigma_{f},\tau_{f}}^{\star}\left(\prod_{f\in\{x_{1},\dots,x_{r}\}}\mathrm{Wg}_{n,\left(|w|_{f}k\right)}\left(\sigma_{f},\tau_{f}\right)\right)\sum_{p_{i}}\sum_{\begin{aligned}J_{f}^{\Sigma}\leftrightarrow\sigma_{f}\\
I_{f}^{\Sigma}\leftrightarrow\tau_{f}
\end{aligned}
}\left(\prod_{f,w}\beta\right).\end{aligned}
\label{eq: expected character with weingarten and beta terms}
\end{equation}
This expression can be further simplified using the projection $\Q_{\lambda,n}.$
Consider Example \ref{exa: expression for expected trace with w=00003D =00005Ba,b=00005D}
-- for all terms fixed in the expression except for e.g. $p_{2},$
the ensuing sum is 
\[
\begin{aligned} & \sum_{p_{2}}\left(\beta_{p_{2},I_{x_{1}}^{1}}\right)\left(\bar{\beta}_{p_{2},J_{x_{2}}^{1}}\right)\\
= & \sum_{p_{2}}\left\langle e_{I_{x_{1}}^{1}},v_{p_{2}}\right\rangle \left\langle v_{p_{2}},e_{J_{x_{2}}^{1}}\right\rangle \\
= & \left\langle \Q_{\lambda,n}\left(e_{I_{x_{1}}^{1}}\right),e_{J_{x_{2}}^{1}}\right\rangle \\
= & \sum_{\pi}^{\leq S_{k}}c(n,k,\lambda,\pi)\left\langle p_{\pi}\left(e_{I_{x_{1}}^{1}}\right),e_{J_{x_{2}}^{1}}\right\rangle .
\end{aligned}
\]
Repeating this and computing the sums over each $p_{i},$ the expression
in Example \ref{exa: expression for expected trace with w=00003D =00005Ba,b=00005D}
becomes 
\begin{equation}
\begin{aligned} & \frac{1}{d_{\lambda}}\sum_{\sigma_{f},\tau_{f}}^{\star}\sum_{\pi_{1},\dots,\pi_{l(w)}}^{\leq S_{k}}\left(\prod_{i=1}^{4}c(n,k,\lambda,\pi_{i})\right)\left(\prod_{f\in\{x_{1},x_{2}\}}\mathrm{Wg}_{n,(|w|_{f}k)}\left(\sigma_{f},\tau_{f}\right)\right)\\
 & \sum_{\sigma_{f}\leftrightarrow J_{f}^{\Sigma},\ \tau_{f}\leftrightarrow I_{f}^{\Sigma}}\left\langle p_{\pi_{1}}\left(e_{I_{x_{1}}^{1}}\right),e_{J_{x_{2}}^{1}}\right\rangle \left\langle p_{\pi_{2}}\left(e_{I_{x_{2}}^{1}}\right),e_{I_{x_{1}}^{2}}\right\rangle \\
 & \left\langle p_{\pi_{3}}\left(e_{J_{x_{1}}^{2}}\right),e_{I_{x_{2}}^{2}}\right\rangle \left\langle p_{\pi_{4}}\left(e_{J_{x_{2}}^{2}}\right),e_{J_{x_{1}}^{1}}\right\rangle .
\end{aligned}
\label{eq: expression for commutator before refining to leq}
\end{equation}
The same argument applies for any non--identity, non--primitive,
cyclically reduced word $w\in F_{r}.$
\begin{defn}
\label{def: N(sigma, tau, pi)}Given, for each $f\in\{x_{1},\dots,x_{r}\}$,
a collection of set partitions $\sigma_{f},\tau_{f}\in\refinedpart$
and a collection of set partitions $\pi_{1},\dots,\pi_{l(w)}\in\skpart,$
we write 
\[
\mathcal{N}\left(\sigma_{x_{1}},\tau_{x_{1}},\dots,\sigma_{x_{r}},\tau_{x_{r}},\pi_{1},\dots,\pi_{l(w)}\right)=\mathcal{N}\left(\sigma_{f},\tau_{f},\pi_{i}\right),
\]
for the number of multi--indices $I_{f}^{1},\dots I_{f}^{|w|_{f}},J_{f}^{1},\dots,J_{f}^{|w|_{f}},$
with all indices distinct, satisfying:
\begin{itemize}
\item $\sigma_{f}\leftrightarrow J_{f}^{\Sigma}$,
\item $\tau_{f}\leftrightarrow I_{f}^{\Sigma}$,
\item $\left\langle p_{\pi_{1}}\left(e_{I_{x_{1}}^{1}}\right),\ e_{(I/J)_{f_{2}}}\right\rangle =1$,
\end{itemize}
$\ \ \ \ \ \vdots$
\begin{itemize}
\item $\left\langle p_{\pi_{l(w)}}\left(e_{(I/J)_{f_{l(w)}}}\right),\ e_{J_{x_{1}}^{1}}\right\rangle =1,$ 
\end{itemize}
where $(I/J)$ is used to denote the correct multi--index arising
from the $\beta$--terms. 
\end{defn}

With this notation, we have proved the following theorem.
\begin{thm}
\label{thm:expected character pi, weingarten, N}Suppose $k\in\mathbb{Z}_{>0}.$
For any $\lambda\vdash k$ and every cyclically reduced, non--identity,
non--primitive word $w\in F_{r},$

\begin{equation}
\begin{aligned} & \mathbb{E}_{w}\left[\chi^{\lambda^{+}(n)}\right]\\
= & \frac{1}{d_{\lambda}}\sum_{\sigma_{f},\tau_{f}}^{\star}\sum_{\pi_{1},\dots,\pi_{l(w)}}^{\leq S_{k}}\left(\prod_{i=1}^{l(w)}c(n,k,\lambda,\pi_{i})\right)\left(\prod_{f\in\{x_{1},\dots,x_{r}\}}\mathrm{Wg}_{n,\left(|w|_{f}k\right)}\left(\sigma_{f},\tau_{f}\right)\right)\mathcal{N}\left(\sigma_{f},\tau_{f},\pi_{i}\right).
\end{aligned}
\label{eq: expected character with pi, weingarten and N}
\end{equation}
\end{thm}

Combining (\ref{eq: expected character with pi, weingarten and N})
with Lemma \ref{rem: leading coeff of pi in projection} and Remark
\ref{rem: leading term for weingarten function}, the following bound
is immediate.
\begin{cor}
\label{cor: first bound for expected character with n^del, n^WG, big N}Suppose
$k\in\mathbb{Z}_{>0}.$ For any $\lambda\vdash k$ and every cyclically
reduced, non--identity, non--primitive word $w\in F_{r},$
\[
\mathbb{E}_{w}\left[\chi^{\lambda^{+}(n)}\right]\ll_{k,l(w)}\sum_{\sigma_{f},\tau_{f}}^{\star}\sum_{\pi_{1},\dots,\pi_{l(w)}}^{\leq S_{k}}n^{-\sum_{i=1}^{l(w)}\mathrm{del}\left(\pi_{i}\right)}n^{-\sum_{f}|\sigma_{f}\wedge\tau_{f}|}\mathcal{N}\left(\sigma_{f},\tau_{f},\pi_{i}\right).
\]
\end{cor}

\subsection{Graphical Interpretation\label{subsec:Graphical-Interpretation}}

We construct a graph for each collection of set partitions $\sigma_{f},\tau_{f},\pi_{1},\dots,\pi_{l(w)}.$
This is similar to the construction of a surface from a matching datum
by Magee in \cite{Magee2021}. Indeed, the graph is essentially the
same as the $1$--skeleton of the surfaces constructed there. 

The graph we construct will be denoted $G\left(\sigma_{f},\tau_{f},\pi_{i}\right).$
The information in (\ref{eq: expected character with pi, weingarten and N})
will be encoded in the properties of $\gsigma,$ allowing us to simplify
further by analyzing the graph rather than dealing with tricky combinatorial
arguments. Ultimately, from $\gsigma$ we will construct a new graph,
from which we derive Theorem \ref{thm: word map main theorem}.

$\gsigma$ is a graph with $2kl(w)$ vertices, separated in to $k$
distinct subsets, each containing exactly $2l(w)$ vertices. We number
the subsets from $1$ to $k$. For each $i\in\{1,\dots,k\}$, the
vertices in subset $i$ are labeled $\left(I_{f}^{1}\right)_{i},\dots\left(I_{f}^{|w|_{f}}\right)_{i},\left(J_{f}^{1}\right)_{i},\dots,\left(J_{f}^{|w|_{f}}\right)_{i}$
for each $f\in\{x_{1},\dots,x_{r}\}$. 

Within each subset $i$, for each $f\in\{x_{1},\dots,x_{r}\}$ and
for each $j\in\{1,\dots,|w|_{f}\},$ we draw a directed, $f$--labeled
edge from the vertex labeled $\left(J_{f}^{j}\right)_{i}$ to the
vertex labeled $\left(I_{f}^{j}\right)_{i}$. This gives a total of
$kl(w)$ directed edges. We refer to the $i^{\mathrm{th}}$ subset
of vertices as the $i^{\mathrm{th}}$ $w$--loop, see Figure \ref{fig: constructing G step 1 - vertices and directed edges only}. 
\begin{center}
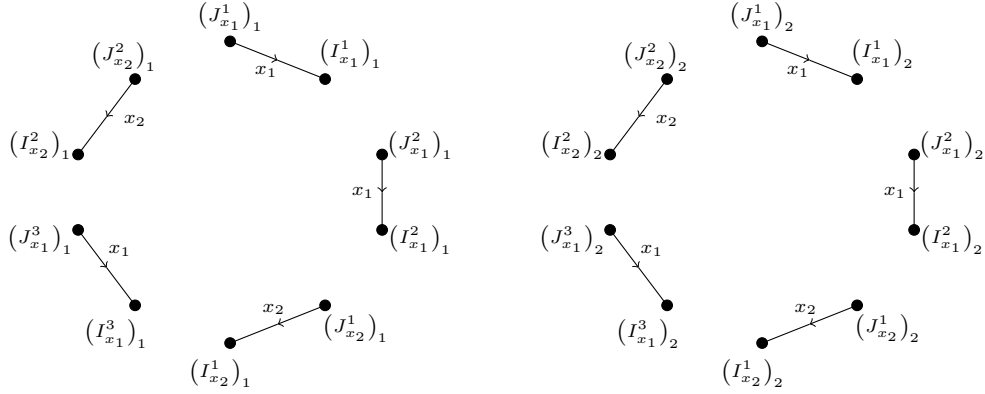
\begin{figure}[h]
\hspace{1.2cm}\begin{tikzpicture}[main/.style = draw, circle]

\node at (11,1.5) {};
\begin{scope}[shift={(-2,0)}]
\filldraw (0,2.5) circle (2pt); \node (1) at (0,2.825) {$\scriptstyle \left(J_{x_1}^1\right)_1$};
\filldraw (1.25,2) circle (2pt); \node (2) at (1.585,2.335) {$\scriptstyle \left(I_{x_1}^1\right)_1$};
\filldraw (2,1) circle (2pt); \node (3) at (2.5,1.125) {$\scriptstyle \left(J_{x_1}^2\right)_1$};
\filldraw (2,0) circle (2pt); \node (4) at (2.5,-0.125) {$\scriptstyle \left(I_{x_1}^2\right)_1$};
\filldraw (1.25,-1) circle (2pt); \node (5) at (1.65,-1.3) {$\scriptstyle \left(J_{x_2}^1\right)_1$};
\filldraw (0,-1.5) circle (2pt); \node (6) at (-0.125,-1.9) {$\scriptstyle \left(I_{x_2}^1\right)_1$};
\filldraw (-1.25,-1) circle (2pt); \node (7) at (-1.5,-1.375) {$\scriptstyle \left(I_{x_1}^3\right)_1$};
\filldraw (-2,0) circle (2pt); \node (8) at (-2.5,-0.125) {$\scriptstyle \left(J_{x_1}^3\right)_1$};
\filldraw (-2,1) circle (2pt); \node (9) at (-2.5,1.125) {$\scriptstyle \left(I_{x_2}^2\right)_1$};
\filldraw (-1.25,2) circle (2pt); \node (10) at (-1.4,2.3) {$\scriptstyle \left(J_{x_2}^2\right)_1$};

\draw[->] (0,2.5) -- (0.625,2.25); \draw  (0.625,2.25)--(1.25,2); \node (21) at (0.475,2.1) {$\scriptstyle x_1$};
\draw[->] (2,1) -- (2,0.5); \draw  (2,0.5)--(2,0); \node (22) at (1.75,0.5) {$\scriptstyle x_1$};
\draw[->] (1.25,-1) -- (0.625,-1.25); \draw  (0.625,-1.25)--(0,-1.5); \node (23) at (0.575,-1.05) {$\scriptstyle x_2$};
\draw (-1.25,-1)--(-1.625,-0.5); \draw[->] (-2,0)-- (-1.625,-0.5); \node (24) at (-1.45,-0.275) {$\scriptstyle x_1$};
\draw (-2,1)--(-1.625,1.5); \draw[->]  (-1.25,2)--(-1.625,1.5); \node (25) at (-1.25,1.45) {$\scriptstyle x_2$};

\filldraw (7,2.5) circle (2pt); \node (11) at (7,2.825) {$\scriptstyle \left(J_{x_1}^1\right)_2$};
\filldraw (8.25,2) circle (2pt); \node (12) at (8.585,2.335) {$\scriptstyle \left(I_{x_1}^1\right)_2$};
\filldraw (9,1) circle (2pt); \node (13) at (9.5,1.125) {$\scriptstyle \left(J_{x_1}^2\right)_2$};
\filldraw (9,0) circle (2pt); \node (14) at (9.5,-0.125) {$\scriptstyle \left(I_{x_1}^2\right)_2$};
\filldraw (8.25,-1) circle (2pt); \node (15) at (8.65,-1.3) {$\scriptstyle \left(J_{x_2}^1\right)_2$};
\filldraw (7,-1.5) circle (2pt); \node (16) at (6.875,-1.9) {$\scriptstyle \left(I_{x_2}^1\right)_2$};
\filldraw (5.75,-1) circle (2pt); \node (17) at (5.5,-1.375) {$\scriptstyle \left(I_{x_1}^3\right)_2$};
\filldraw (5,0) circle (2pt); \node (18) at (4.5,-0.125) {$\scriptstyle \left(J_{x_1}^3\right)_2$};
\filldraw (5,1) circle (2pt); \node (19) at (4.5,1.125) {$\scriptstyle \left(I_{x_2}^2\right)_2$};
\filldraw (5.75,2) circle (2pt); \node (20) at (5.6,2.3) {$\scriptstyle \left(J_{x_2}^2\right)_2$};

\draw[->] (7,2.5) -- (7.625,2.25); \draw  (7.625,2.25)--(8.25,2); \node (26) at (7.475,2.1) {$\scriptstyle x_1$};
\draw[->] (9,1) -- (9,0.5); \draw  (9,0.5)--(9,0); \node (27) at (8.75,0.5) {$\scriptstyle x_1$};
\draw[->] (8.25,-1) -- (7.625,-1.25); \draw  (7.625,-1.25)--(7,-1.5); \node (28) at (7.575,-1.05) {$\scriptstyle x_2$};
\draw (5.75,-1) -- (5.375,-0.5); \draw[->] (5,0)-- (5.375,-0.5); \node (29) at (5.55,-0.275) {$\scriptstyle x_1$};
\draw (5,1) -- (5.375,1.5); \draw[->] (5.75,2)-- (5.375,1.5); \node (30) at (5.75,1.45) {$\scriptstyle x_2$};
\end{scope}

\end{tikzpicture}

\caption{\label{fig: constructing G step 1 - vertices and directed edges only}Here
we have shown how to begin constructing $\protect\gsigma$ with $w=x_{1}x_{1}x_{2}x_{1}^{-1}x_{2}^{-1}$
and $k=2.$ There are $2kl(w)=20$ vertices, split in to $k=2$ distinct
subsets, each containing $2l(w)=10$ vertices each. The vertices on
the left (i.e. with the outer subscript ``$1$'') are the $1^{\mathrm{st}}$
$w$--loop and the vertices on the right are the $2^{\mathrm{nd}}$
$w$--loop. There are $k|w|_{x_{1}}=6$ directed $x_{1}$--edges
and $k|w|_{x_{2}}=4$ directed $x_{2}$--edges.}
\end{figure}
\par\end{center}

For each $f\in\{x_{1},\dots,x_{r}\}$, we add an undirected $\sigma_{f}$--edge
between any two vertices $\left(J_{f}^{j}\right)_{i}$ and $\left(J_{f}^{j'}\right)_{i'}$
whenever $\sigma_{f}$ dictates that the corresponding indices must
be equal in order for $\delta_{\sigma_{f}}\left(J_{f}^{1}\sqcup J_{f}^{2}\sqcup\dots\sqcup J_{f}^{|w|_{f}}\right)$
to be non--zero. This is illustrated in Figure \ref{fig: consturcting graph with vertices, directed edges and sigma edges}. 

\begin{figure}[H]
\vspace{-0.5cm}

\begin{tikzpicture}[main/.style = draw, circle]


\draw [red]  (0,2.5) to (2,1); \draw [red] (2,1) to (5,0); \draw [red] (0,2.5) [out=30] to (5,0);
\draw [red] (7,2.5) to (9,1); \draw [red] (-2,0) .. controls (-4.5,4) and (8,6) ..  (9,1); \draw[red] (-2,0) .. controls (4.5, 3)..  (7,2.5);

\draw [violet] (-1.25,2) to (1.25,-1); \draw[violet] (5.75,2) to (8.25,-1);

\filldraw (0,2.5) circle (2pt); \node (1) at (0,2.825) {$\scriptstyle \left(J_{x_1}^1\right)_1$};
\filldraw (1.25,2) circle (2pt); \node (2) at (1.585,2.335) {$\scriptstyle \left(I_{x_1}^1\right)_1$};
\filldraw (2,1) circle (2pt); \node (3) at (2.5,1.125) {$\scriptstyle \left(J_{x_1}^2\right)_1$};
\filldraw (2,0) circle (2pt); \node (4) at (2.5,-0.125) {$\scriptstyle \left(I_{x_1}^2\right)_1$};
\filldraw (1.25,-1) circle (2pt); \node (5) at (1.65,-1.3) {$\scriptstyle \left(J_{x_2}^1\right)_1$};
\filldraw (0,-1.5) circle (2pt); \node (6) at (-0.125,-1.9) {$\scriptstyle \left(I_{x_2}^1\right)_1$};
\filldraw (-1.25,-1) circle (2pt); \node (7) at (-1.5,-1.375) {$\scriptstyle \left(I_{x_1}^3\right)_1$};
\filldraw (-2,0) circle (2pt); \node (8) at (-2.5,-0.125) {$\scriptstyle \left(J_{x_1}^3\right)_1$};
\filldraw (-2,1) circle (2pt); \node (9) at (-2.5,1.125) {$\scriptstyle \left(I_{x_2}^2\right)_1$};
\filldraw (-1.25,2) circle (2pt); \node (10) at (-1.4,2.3) {$\scriptstyle \left(J_{x_2}^2\right)_1$};

\draw[->] (0,2.5) -- (0.625,2.25); \draw  (0.625,2.25)--(1.25,2); \node (21) at (0.475,2.1) {$\scriptstyle x_1$};
\draw[->] (2,1) -- (2,0.5); \draw  (2,0.5)--(2,0); \node (22) at (1.75,0.5) {$\scriptstyle x_1$};
\draw[->] (1.25,-1) -- (0.625,-1.25); \draw  (0.625,-1.25)--(0,-1.5); \node (23) at (0.575,-1.05) {$\scriptstyle x_2$};
\draw (-1.25,-1)--(-1.625,-0.5); \draw[->] (-2,0)-- (-1.625,-0.5); \node (24) at (-1.45,-0.275) {$\scriptstyle x_1$};
\draw (-2,1)--(-1.625,1.5); \draw[->]  (-1.25,2)--(-1.625,1.5); \node (25) at (-1.25,1.45) {$\scriptstyle x_2$};

\filldraw (7,2.5) circle (2pt); \node (11) at (7,2.825) {$\scriptstyle \left(J_{x_1}^1\right)_2$};
\filldraw (8.25,2) circle (2pt); \node (12) at (8.585,2.335) {$\scriptstyle \left(I_{x_1}^1\right)_2$};
\filldraw (9,1) circle (2pt); \node (13) at (9.5,1.125) {$\scriptstyle \left(J_{x_1}^2\right)_1$};
\filldraw (9,0) circle (2pt); \node (14) at (9.5,-0.125) {$\scriptstyle \left(I_{x_1}^2\right)_2$};
\filldraw (8.25,-1) circle (2pt); \node (15) at (8.65,-1.3) {$\scriptstyle \left(J_{x_2}^1\right)_2$};
\filldraw (7,-1.5) circle (2pt); \node (16) at (6.875,-1.9) {$\scriptstyle \left(I_{x_2}^1\right)_2$};
\filldraw (5.75,-1) circle (2pt); \node (17) at (5.5,-1.375) {$\scriptstyle \left(I_{x_1}^3\right)_2$};
\filldraw (5,0) circle (2pt); \node (18) at (4.5,-0.125) {$\scriptstyle \left(J_{x_1}^3\right)_2$};
\filldraw (5,1) circle (2pt); \node (19) at (4.5,1.125) {$\scriptstyle \left(I_{x_2}^2\right)_2$};
\filldraw (5.75,2) circle (2pt); \node (20) at (5.6,2.3) {$\scriptstyle \left(J_{x_2}^2\right)_2$};

\draw[->] (7,2.5) -- (7.625,2.25); \draw  (7.625,2.25)--(8.25,2); \node (26) at (7.475,2.1) {$\scriptstyle x_1$};
\draw[->] (9,1) -- (9,0.5); \draw  (9,0.5)--(9,0); \node (27) at (8.75,0.5) {$\scriptstyle x_1$};
\draw[->] (8.25,-1) -- (7.625,-1.25); \draw  (7.625,-1.25)--(7,-1.5); \node (28) at (7.575,-1.05) {$\scriptstyle x_2$};
\draw (5.75,-1) -- (5.375,-0.5); \draw[->] (5,0)-- (5.375,-0.5); \node (29) at (5.55,-0.275) {$\scriptstyle x_1$};
\draw (5,1) -- (5.375,1.5); \draw[->] (5.75,2)-- (5.375,1.5); \node (30) at (5.75,1.45) {$\scriptstyle x_2$}; 

\end{tikzpicture} 

\caption{\label{fig: consturcting graph with vertices, directed edges and sigma edges}Here
we continue the construction of the graph from Figure \ref{fig: constructing G step 1 - vertices and directed edges only}
by adding in the $\sigma_{x_{1}}$--edges (in red) and the $\sigma_{x_{2}}$--edges
(in purple). In this example, we have $\sigma_{x_{1}}=\big\{\{1,3,6\},\ \{2,4,5\}\big\}\in\mathrm{Part}\left([3k]\right]$
and $\sigma_{x_{2}}=\big\{\{1,3\},\ \{2,4\}\big\}\in\protect\part$.}
\end{figure}
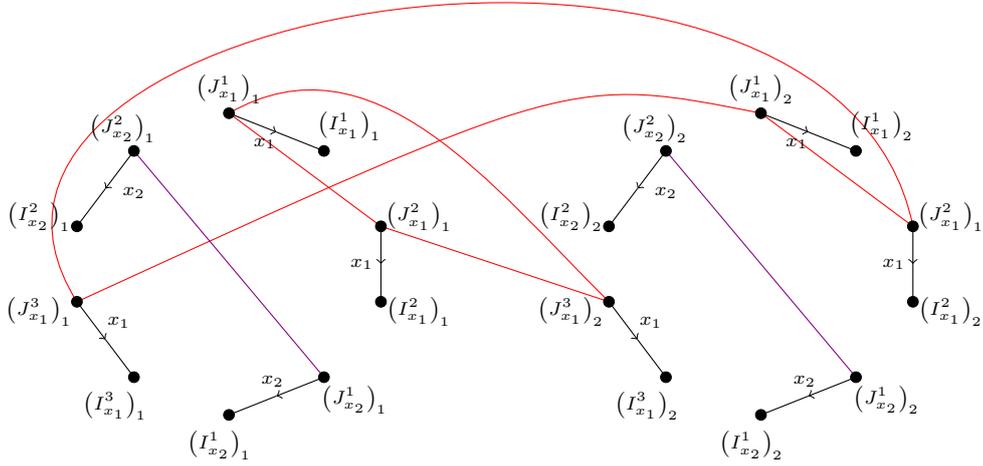

For each $f\in\{x_{1},\dots,x_{r}\}$, we add undirected $\tau_{f}$--edges
similarly, illustrated in Figure \ref{fig: construction of graph with vertices, directed edges, sigma and tau edges}.
Finally, for each $i\in\{1,\dots,l(w)\},$ we add an undirected $\pi_{i}$--edge
between any two vertices whenever $\pi_{i}$ dictates that the\emph{
}corresponding indices must be equal\emph{ }in order for the corresponding
inner product term in Definition \ref{def: N(sigma, tau, pi)} to
be non--zero. 

\emph{In short, for each collection of set partitions $\sigma_{f},\tau_{f},\pi_{1},\dots,\pi_{l(w)},$
we draw an undirected edge between any two vertices for which the
indices with the same label are necessarily equal for the conditions
in Definition \ref{def: N(sigma, tau, pi)} to hold. }

Figure \ref{fig: final construction of gsigma} shows a complete example
of $\gsigma$, with each $\sigma_{f},$ $\tau_{f}$ and $\pi_{i}$
as described in the examples throughout this section. 

\begin{figure}[H]
\begin{tikzpicture}[main/.style = draw, circle]


\draw [red]  (0,2.5) to (2,1); \draw [red] (2,1) to (5,0); \draw [red] (0,2.5) [out=30] to (5,0);
\draw [red] (7,2.5) to (9,1); \draw [red] (-2,0) .. controls (-4.5,4) and (8,6) ..  (9,1); \draw[red] (-2,0) .. controls (4.5, 3)..  (7,2.5);

\draw [violet] (-1.25,2) to (1.25,-1); \draw[violet] (5.75,2) to (8.25,-1);

\draw [blue] (1.25,2) to (2,0); \draw [blue] (2,0) to (-1.25,-1); \draw [blue] (-1.25,-1) to (1.25,2);
\draw [blue] (8.25,2) to (9,0); \draw [blue] (9,0) to (5.75,-1); \draw [blue] (5.75,-1) to (8.25,2);

\draw [cyan] (0,-1.5) .. controls (2,-2.75) and (3.3,0.5) .. (5,1); \draw[cyan] (-2,1) .. controls (-3.5,-1.75) and (3, -4) .. (7,-1.5);

\filldraw (0,2.5) circle (2pt); \node (1) at (0,2.825) {$\scriptstyle \left(J_{x_1}^1\right)_1$};
\filldraw (1.25,2) circle (2pt); \node (2) at (1.585,2.335) {$\scriptstyle \left(I_{x_1}^1\right)_1$};
\filldraw (2,1) circle (2pt); \node (3) at (2.5,1.125) {$\scriptstyle \left(J_{x_1}^2\right)_1$};
\filldraw (2,0) circle (2pt); \node (4) at (2.5,-0.125) {$\scriptstyle \left(I_{x_1}^2\right)_1$};
\filldraw (1.25,-1) circle (2pt); \node (5) at (1.65,-1.3) {$\scriptstyle \left(J_{x_2}^1\right)_1$};
\filldraw (0,-1.5) circle (2pt); \node (6) at (-0.125,-1.9) {$\scriptstyle \left(I_{x_2}^1\right)_1$};
\filldraw (-1.25,-1) circle (2pt); \node (7) at (-1.5,-1.375) {$\scriptstyle \left(I_{x_1}^3\right)_1$};
\filldraw (-2,0) circle (2pt); \node (8) at (-2.5,-0.125) {$\scriptstyle \left(J_{x_1}^3\right)_1$};
\filldraw (-2,1) circle (2pt); \node (9) at (-2.5,1.125) {$\scriptstyle \left(I_{x_2}^2\right)_1$};
\filldraw (-1.25,2) circle (2pt); \node (10) at (-1.4,2.3) {$\scriptstyle \left(J_{x_2}^2\right)_1$};

\draw[->] (0,2.5) -- (0.625,2.25); \draw  (0.625,2.25)--(1.25,2); \node (21) at (0.475,2.1) {$\scriptstyle x_1$};
\draw[->] (2,1) -- (2,0.5); \draw  (2,0.5)--(2,0); \node (22) at (1.75,0.5) {$\scriptstyle x_1$};
\draw[->] (1.25,-1) -- (0.625,-1.25); \draw  (0.625,-1.25)--(0,-1.5); \node (23) at (0.575,-1.05) {$\scriptstyle x_2$};
\draw (-1.25,-1)--(-1.625,-0.5); \draw[->] (-2,0)-- (-1.625,-0.5); \node (24) at (-1.45,-0.275) {$\scriptstyle x_1$};
\draw (-2,1)--(-1.625,1.5); \draw[->]  (-1.25,2)--(-1.625,1.5); \node (25) at (-1.25,1.45) {$\scriptstyle x_2$};

\filldraw (7,2.5) circle (2pt); \node (11) at (7,2.825) {$\scriptstyle \left(J_{x_1}^1\right)_2$};
\filldraw (8.25,2) circle (2pt); \node (12) at (8.585,2.335) {$\scriptstyle \left(I_{x_1}^1\right)_2$};
\filldraw (9,1) circle (2pt); \node (13) at (9.5,1.125) {$\scriptstyle \left(J_{x_1}^2\right)_1$};
\filldraw (9,0) circle (2pt); \node (14) at (9.5,-0.125) {$\scriptstyle \left(I_{x_1}^2\right)_2$};
\filldraw (8.25,-1) circle (2pt); \node (15) at (8.65,-1.3) {$\scriptstyle \left(J_{x_2}^1\right)_2$};
\filldraw (7,-1.5) circle (2pt); \node (16) at (6.875,-1.9) {$\scriptstyle \left(I_{x_2}^1\right)_2$};
\filldraw (5.75,-1) circle (2pt); \node (17) at (5.5,-1.375) {$\scriptstyle \left(I_{x_1}^3\right)_2$};
\filldraw (5,0) circle (2pt); \node (18) at (4.5,-0.125) {$\scriptstyle \left(J_{x_1}^3\right)_2$};
\filldraw (5,1) circle (2pt); \node (19) at (4.5,1.125) {$\scriptstyle \left(I_{x_2}^2\right)_2$};
\filldraw (5.75,2) circle (2pt); \node (20) at (5.6,2.3) {$\scriptstyle \left(J_{x_2}^2\right)_2$};

\draw[->] (7,2.5) -- (7.625,2.25); \draw  (7.625,2.25)--(8.25,2); \node (26) at (7.475,2.1) {$\scriptstyle x_1$};
\draw[->] (9,1) -- (9,0.5); \draw  (9,0.5)--(9,0); \node (27) at (8.75,0.5) {$\scriptstyle x_1$};
\draw[->] (8.25,-1) -- (7.625,-1.25); \draw  (7.625,-1.25)--(7,-1.5); \node (28) at (7.575,-1.05) {$\scriptstyle x_2$};
\draw (5.75,-1) -- (5.375,-0.5); \draw[->] (5,0)-- (5.375,-0.5); \node (29) at (5.55,-0.275) {$\scriptstyle x_1$};
\draw (5,1) -- (5.375,1.5); \draw[->] (5.75,2)-- (5.375,1.5); \node (30) at (5.75,1.45) {$\scriptstyle x_2$};    \end{tikzpicture}

\caption{\label{fig: construction of graph with vertices, directed edges, sigma and tau edges}We
continue the construction from Figures \ref{fig: constructing G step 1 - vertices and directed edges only}
and \ref{fig: consturcting graph with vertices, directed edges and sigma edges}.
We have added in the $\tau_{x_{1}}$--edges (in dark blue) and the
$\tau_{x_{2}}$--edges (in light blue). In this example, $\tau_{x_{1}}=\big\{\{1,3,5\}\ \{2,4,6\}\big\}$
and $\tau_{x_{2}}=\big\{\{1,4\}\ \{2,3\}\big\}$.}

\begin{tikzpicture}[main/.style = draw, circle]

\draw [red]  (0,2.5) to (2,1); \draw [red] (2,1) to (5,0); \draw [red] (0,2.5) [out=30] to (5,0);
\draw [red] (7,2.5) to (9,1); \draw [red] (-2,0) .. controls (-4.5,4) and (8,6) ..  (9,1); \draw[red] (-2,0) .. controls (4.5, 3)..  (7,2.5);

\draw [violet] (-1.25,2) to (1.25,-1); \draw[violet] (5.75,2) to (8.25,-1);

\draw [blue] (1.25,2) to (2,0); \draw [blue] (2,0) to (-1.25,-1); \draw [blue] (-1.25,-1) to (1.25,2);
\draw [blue] (8.25,2) to (9,0); \draw [blue] (9,0) to (5.75,-1); \draw [blue] (5.75,-1) to (8.25,2);

\draw [cyan] (0,-1.5) .. controls (2,-2.75) and (3.3,0.5) .. (5,1); \draw[cyan] (-2,1) .. controls (-3.5,-1.75) and (3, -4) .. (7,-1.5);

\draw [green] (2,0) to (1.25,-1); \draw [green] (9,0) to (8.25,-1);
\draw [green] (0,-1.5) .. controls (0.5,-3) and (3,-2) .. (5.75,-1);
\draw [green] (5,0) to (5,1);
\draw [green] (0,2.5) .. controls (3,3.3) .. (5.75,2); \draw [green] (-1.25,2) .. controls (0,4.5) and (5.75,4) .. (7,2.5);

\filldraw (0,2.5) circle (2pt); \node (1) at (0,2.825) {$\scriptstyle \left(J_{x_1}^1\right)_1$};
\filldraw (1.25,2) circle (2pt); \node (2) at (1.585,2.335) {$\scriptstyle \left(I_{x_1}^1\right)_1$};
\filldraw (2,1) circle (2pt); \node (3) at (2.5,1.125) {$\scriptstyle \left(J_{x_1}^2\right)_1$};
\filldraw (2,0) circle (2pt); \node (4) at (2.5,-0.125) {$\scriptstyle \left(I_{x_1}^2\right)_1$};
\filldraw (1.25,-1) circle (2pt); \node (5) at (1.65,-1.3) {$\scriptstyle \left(J_{x_2}^1\right)_1$};
\filldraw (0,-1.5) circle (2pt); \node (6) at (-0.125,-1.9) {$\scriptstyle \left(I_{x_2}^1\right)_1$};
\filldraw (-1.25,-1) circle (2pt); \node (7) at (-1.5,-1.375) {$\scriptstyle \left(I_{x_1}^3\right)_1$};
\filldraw (-2,0) circle (2pt); \node (8) at (-2.5,-0.125) {$\scriptstyle \left(J_{x_1}^3\right)_1$};
\filldraw (-2,1) circle (2pt); \node (9) at (-2.5,1.125) {$\scriptstyle \left(I_{x_2}^2\right)_1$};
\filldraw (-1.25,2) circle (2pt); \node (10) at (-1.4,2.3) {$\scriptstyle \left(J_{x_2}^2\right)_1$};

\draw[->] (0,2.5) -- (0.625,2.25); \draw  (0.625,2.25)--(1.25,2); \node (21) at (0.475,2.1) {$\scriptstyle x_1$};
\draw[->] (2,1) -- (2,0.5); \draw  (2,0.5)--(2,0); \node (22) at (1.75,0.5) {$\scriptstyle x_1$};
\draw[->] (1.25,-1) -- (0.625,-1.25); \draw  (0.625,-1.25)--(0,-1.5); \node (23) at (0.575,-1.05) {$\scriptstyle x_2$};
\draw (-1.25,-1)--(-1.625,-0.5); \draw[->] (-2,0)-- (-1.625,-0.5); \node (24) at (-1.45,-0.275) {$\scriptstyle x_1$};
\draw (-2,1)--(-1.625,1.5); \draw[->]  (-1.25,2)--(-1.625,1.5); \node (25) at (-1.25,1.45) {$\scriptstyle x_2$};

\filldraw (7,2.5) circle (2pt); \node (11) at (7,2.825) {$\scriptstyle \left(J_{x_1}^1\right)_2$};
\filldraw (8.25,2) circle (2pt); \node (12) at (8.585,2.335) {$\scriptstyle \left(I_{x_1}^1\right)_2$};
\filldraw (9,1) circle (2pt); \node (13) at (9.5,1.125) {$\scriptstyle \left(J_{x_1}^2\right)_1$};
\filldraw (9,0) circle (2pt); \node (14) at (9.5,-0.125) {$\scriptstyle \left(I_{x_1}^2\right)_2$};
\filldraw (8.25,-1) circle (2pt); \node (15) at (8.65,-1.3) {$\scriptstyle \left(J_{x_2}^1\right)_2$};
\filldraw (7,-1.5) circle (2pt); \node (16) at (6.875,-1.9) {$\scriptstyle \left(I_{x_2}^1\right)_2$};
\filldraw (5.75,-1) circle (2pt); \node (17) at (5.5,-1.375) {$\scriptstyle \left(I_{x_1}^3\right)_2$};
\filldraw (5,0) circle (2pt); \node (18) at (4.5,-0.125) {$\scriptstyle \left(J_{x_1}^3\right)_2$};
\filldraw (5,1) circle (2pt); \node (19) at (4.5,1.125) {$\scriptstyle \left(I_{x_2}^2\right)_2$};
\filldraw (5.75,2) circle (2pt); \node (20) at (5.6,2.3) {$\scriptstyle \left(J_{x_2}^2\right)_2$};

\draw[->] (7,2.5) -- (7.625,2.25); \draw  (7.625,2.25)--(8.25,2); \node (26) at (7.475,2.1) {$\scriptstyle x_1$};
\draw[->] (9,1) -- (9,0.5); \draw  (9,0.5)--(9,0); \node (27) at (8.75,0.5) {$\scriptstyle x_1$};
\draw[->] (8.25,-1) -- (7.625,-1.25); \draw  (7.625,-1.25)--(7,-1.5); \node (28) at (7.575,-1.05) {$\scriptstyle x_2$};
\draw (5.75,-1) -- (5.375,-0.5); \draw[->] (5,0)-- (5.375,-0.5); \node (29) at (5.55,-0.275) {$\scriptstyle x_1$};
\draw (5,1) -- (5.375,1.5); \draw[->] (5.75,2)-- (5.375,1.5); \node (30) at (5.75,1.45) {$\scriptstyle x_2$};\end{tikzpicture}

\caption{\label{fig: final construction of gsigma}This depicts $\protect\gsigma,$
where $\sigma_{f}$ and $\tau_{f}$ are as described in Figures \ref{fig: consturcting graph with vertices, directed edges and sigma edges}
and \ref{fig: construction of graph with vertices, directed edges, sigma and tau edges}
and the $\pi_{i}$ are as follows: $\pi_{1}=\big\{\{1\},\ \{2\},\ \{3\}\ \{4\}\big\},$
$\pi_{2}=\big\{\{1,3\},\ \{2,4\}\big\}$, $\pi_{3}=\big\{\{1,4\},\ \{2\},\ \{3\}\big\}$,
$\pi_{4}=\big\{\{1\},\ \{3\},\ \{2,4\}\big\}$ and $\pi_{5}=\big\{\{1,4\},\ \{2,3\}\big\}$.}
\end{figure}

\subsection{\label{subsec:Obtaining-the-bound}Obtaining the bound}

With all $\pi_{1},\dots,\pi_{l(w)}$ corresponding to permutations
(i.e. with $\sum_{i}\mathrm{del}(\pi_{i})=0$), we can apply \cite[Theorem 1.2]{LouderWilton}
after (\ref{eq: expected character wth gamsig TILDE}) below to obtain
the bound of $n^{-k}.$ For the other cases, we need to prove a variant
of this theorem that applies in our case. This section is dedicated
to this task. 

\subsubsection{Stackings}

Given graph morphisms $\rho_{1}:\Gamma_{1}\to G$ and $\rho_{2}:\Gamma_{2}\to G$,
the fibre product $\Gamma_{1}\times_{G}\Gamma_{2}$ is the graph with
vertex set 
\[
V\left(\Gamma_{1}\times_{G}\Gamma_{2}\right)=\big\{(v_{1},v_{2})\in V\left(\Gamma_{1}\right)\times V\left(\Gamma_{2}\right):\ \rho_{1}\left(v_{1}\right)=\rho_{2}\left(v_{2}\right)\big\}
\]
and edge set 
\[
E\left(\Gamma_{1}\times_{G}\Gamma_{2}\right)=\big\{(e_{1},e_{2})\in E\left(\Gamma_{1}\right)\times E\left(\Gamma_{2}\right):\ \rho_{1}\left(e_{1}\right)=\rho_{2}\left(e_{2}\right)\big\},
\]
where $t\left(e_{1},e_{2}\right)=\left(t\left(e_{1}\right),t\left(e_{2}\right)\right)$
and $h\left(e_{1},e_{2}\right)=\left(h\left(e_{1}\right),h\left(e_{2}\right)\right).$ 

Louder and Wilton \cite{LouderWilton} first developed the notion
of a stacking of a graph immersion and we adapt their definitions
slightly to suit our case.
\begin{defn}
Let $\Gamma$ be a finite graph and $\mathbb{S}$ a $1$--complex
with an immersion $\Lambda:\mathbb{S}\to\Gamma.$ A \emph{stacking
}is an embedding $\hat{\Lambda}:\mathbb{S}\to\Gamma\times\mathbb{R}$
such that $\pi\hat{\Lambda}=\Lambda,$ where $\pi:\Gamma\times\mathbb{R}\to\Gamma$
is the trivial $\mathbb{R}$--bundle. 
\end{defn}

Let $\eta:\Gamma\times\mathbb{R}\to\mathbb{R}$ be the projection
to $\mathbb{R}.$ 
\begin{defn}
Given a stacking $\hat{\Lambda}:\mathbb{S}\to\Gamma\times\mathbb{R}$
of an immersion $\Lambda:\mathbb{S}\to\Gamma,$ define 
\[
\mathcal{A}_{\hat{\Lambda}}\overset{\mathrm{def}}{=}\big\{ x\in\mathbb{S}:\ \forall y\neq x,\ \mathrm{if}\ \Lambda(x)=\Lambda(y)\ \mathrm{then}\ \eta\left(\hat{\Lambda}(x)\right)>\eta\left(\hat{\Lambda}(y)\right)\big\}
\]
and 
\[
\mathcal{B}_{\hat{\Lambda}}\overset{\mathrm{def}}{=}\big\{ x\in\mathbb{S}:\ \forall y\neq x,\ \mathrm{if}\ \Lambda(x)=\Lambda(y)\ \mathrm{then}\ \eta\left(\hat{\Lambda}(x)\right)<\eta\left(\hat{\Lambda}(y)\right)\big\}.
\]
\end{defn}

\subsubsection{Bounding expected character}

By Remark \ref{rem: main bound holds for proper powers}, we only
need to prove the bound in Theorem \ref{thm: word map main theorem}
in the case where $w$ is not a proper power. \emph{From now, assume
that $w$ is not a proper power. }

For each $\sigma_{f},\tau_{f},\pi_{i},$ construct the graph $\gsigma$
from Section \ref{subsec:Graphical-Interpretation}, consisting of
$kl(w)$ disjoint, directed $f$ edges organised into $k$ $w$--loops
and then marking on the blocks of the partitions $\sigma_{f},\tau_{f},\pi_{i}$
using undirected edges. 

We construct a new graph $\gamsig$ by: 
\begin{itemize}
\item gluing together any vertices that are connected by a $\sigma_{f},\tau_{f}$
or $\pi_{i}$--edge (and then deleting the $\sigma_{f},\tau_{f}$
or $\pi_{i}$--edges),
\item if $\sigma_{f}$ and $\tau_{f}$ connect both the initial and terminal
vertices of some collection of $f$--edges, we merge these into a
single $f$--edge, see Figure \ref{fig: constructing tilde gamsig and GAMMA}. 
\end{itemize}
\begin{defn}
We will write $\hatgsigma$ for the subgraph of $\gsigma$ consisting
of every vertex and \emph{only the undirected edges}. 
\end{defn}

\begin{lem}
Given $\sigma_{f},\tau_{f}\in\refinedpart$ and $\pi_{i}\in\skpart,$
\[
\sum_{f}|\sigma_{f}\wedge\tau_{f}|=\Big|E\left(\gamsig\right)\Big|.
\]
 
\end{lem}

\begin{proof}
Every block of $\sigma_{f}\wedge\tau_{f}$ of size $p$ corresponds
to some collection of $f$--edges in $\gsigma$ of size $p$, whose
initial and terminal vertices have been glued together and whose edges
have been merged in the construction of $\gamsig.$ Every $f$--edge
in $\gamsig$ then corresponds to a block of $\sigma_{f}\wedge\tau_{f}.$ 
\end{proof}
\begin{lem}
Given $\sigma_{f},\tau_{f}\in\refinedpart$ and $\pi_{i}\in\skpart,$
\[
\mathcal{N}(\sigma_{f},\tau_{f},\pi_{i})\ll n^{\left|V\left(\gamsig\right)\right|}.
\]
\end{lem}

\begin{proof}
Each vertex of $\gamsig$ corresponds to a connected component of
$\hatgsigma$. Any collection of indices must be equal if their corresponding
vertices are in the same connected component of $\hatgsigma$. Clearly,
if there were no more restrictions, we would have 
\[
\mathcal{N}(\sigma_{f},\tau_{f},\pi_{i})=n^{\left|V\left(\gamsig\right)\right|}.
\]
 The lemma follows since we have the additional restriction that within
each multi--index, each index must be distinct. 
\end{proof}
So (\ref{eq: expected character with pi, weingarten and N}) becomes
\begin{equation}
\mathbb{E}_{w}\left[\chi^{\lambda^{+}(n)}\right]\ll_{k,l(w)}\sum_{\sigma_{f},\tau_{f}}^{\star}\sum_{\pi_{1},\dots,\pi_{l(w)}}^{\leq S_{k}}n^{-\sum_{i}\mathrm{del}(\pi_{i})}n^{\chi\left(\gamsig\right)}.\label{eq: expected character wth gamsig TILDE}
\end{equation}

\begin{figure}[H]
\tikzset{->-/.style={decoration={   markings,   mark=at position .5 with {\arrow{>}}},postaction={decorate}}} \begin{tikzpicture}[main/.style = draw, circle]

\hspace{2.9cm}
\filldraw (0,0) circle (2pt);  \filldraw (4,0) circle (2pt);  \filldraw (2,-1.5) circle (2pt); \node at (2,-1.8) {$\scriptstyle \left[I_{x_{2}}^{2}\right]_{1}=\scriptstyle \left[I_{x_{2}}^{1}\right]_{2}$};
\draw[->-] (0,0) ..controls (1.5,1) and (2.5,1) .. (4,0); \node at (2,1) {$\scriptstyle x_1 $}; \draw[->-] (4,0)-- (0,0); \node at (2,0.2) {$\scriptstyle x_2 $}; \draw[->-] (4,0) ..controls (2.5,-0.5) and (1.5,-0.5) .. (0,0); \node at (2,-0.65) {$\scriptstyle x_1 $}; \draw[->-] (4,0) ..controls (4,-0.3) and (2.3,-1.5) .. (2,-1.5); \node at (3.3,-1.1) {$\scriptstyle x_2 $}; \draw[->-] (0,0) ..controls (0,-0.3) and (1.7,-1.5) .. (2,-1.5); \node at (0.7,-1.1) {$\scriptstyle x_2 $};
\draw[->-] (4,0) ..controls (5.5,1.5) and (5.5,-1.5) .. (4,0); \node at (5.4,0) {$\scriptstyle x_1 $}; \draw[->-] (0,0) ..controls (-2,0) and (0,2) .. (0,0); \node at (-0.9,0.9) {$\scriptstyle x_2 $}; \draw[->-] (0,0) ..controls (-2,0) and (0,-2) .. (0,0); \node at (-0.9,-0.9) {$\scriptstyle x_1 $};

\end{tikzpicture}

\caption{\label{fig: constructing tilde gamsig and GAMMA}Above we show how
to construct $\protect\gamsig$ from the graph $\protect\gsigma$
constructed in Figure \ref{fig: final construction of gsigma}. We
have labeled one vertex to show which vertices of $\protect\gsigma$
have been glued together in the construction.}
\end{figure}
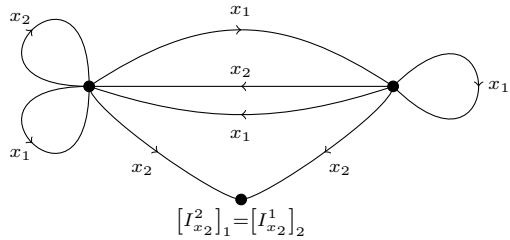

\begin{rem}
It is worth noting here why the most obvious approach to obtaining
the required bound does not work. We can `complete' each $\pi_{i}$
to a permutation $\hat{\pi}_{i}$ by gluing together vertices of $\gamsig$
according to the additional identifications dictated by the edges
that must be added to $\pi_{i}$ to construct $\hat{\pi}_{i}.$ Then,
one may observe that 
\[
n^{-\sum_{i}\mathrm{del}(\pi_{i})}n^{\chi\left(\gamsig\right)}\ll n^{\chi\left(\Gamma\left(\sigma_{f},\tau_{f},\hat{\pi}_{i}\right)\right)}
\]
and seek to apply \cite[Theorem 1.2]{LouderWilton} directly. Unfortunately,
in gluing the vertices of $\gamsig$ together, one may lose the distinctness
property that no two vertices of $\gsigma$ labeled, say, $\left(J_{f}^{i}\right)_{j_{1}}$
and $\left(J_{f}^{i}\right)_{j_{2}}$ with $j_{1}\neq j_{2}$, can
be glued together, which would then weaken the bound achieved by applying
\cite[Theorem 1.2]{LouderWilton}. Therefore, we proceed with our
extension of their result. 
\end{rem}

Define $\mathcal{B}_{r}\overset{\mathrm{def}}{=}\bigvee_{i=1}^{r}S_{i}^{1},$
the bouquet of $r$ oriented circles, labeled $x_{1},\dots,x_{r}$
and with wedge point labeled $o$. Define $W$ to be the graph consisting
of a single cycle of oriented $f$--labeled edges that, when traversed
in an appropriate direction from an appropriate vertex, reads out
the word $w.$ Then $W$ comes equipped with an obvious primitive
(since $w$ is not a proper power) immersion 
\[
\Lambda:W\to\mathcal{B}_{r}.
\]
 By \cite[Lemma 3.4]{LouderWilton}, this immersion has a stacking
\[
\hat{\Lambda}:W\to\mathcal{B}_{r}\times\mathbb{R}.
\]
 For each $\sigma_{f},\tau_{f},\pi_{i},$ write $\Gamma=\gamsig.$
We have a map 
\[
\rho:\Gamma\to\mathcal{B}_{r}
\]
and we can form the fibre product graph 
\[
\Gamma\times_{\mathcal{B}_{r}}W.
\]

We denote by $\mathbb{S}$ the components of the fibre product graph
defined by following the paths in $\gsigma$ which alternate between
$f$--edges and $\pi_{i}$--edges. This defines some partition of
$w^{k}.$ The connected components of $\mathbb{S}$ are either \emph{circles
}that read out $w^{k_{i}}$ or \emph{pieces, }which consist of two
closed endpoints and a chain of vertices of valence two in between
them. Each piece reads out some subword of $w^{k}.$ The concatenation
of the words spelled out by all the connected components of $\mathbb{S}$
is exactly $w^{k}$ and the number of pieces of $\mathbb{S}$ is exactly
$\sum_{i}\mathrm{del}\left(\pi_{i}\right).$ 

We have an immersion 
\[
\Lambda':\mathbb{S}\to\Gamma
\]
and an immersion 
\[
\delta:\mathbb{S}\to W.
\]
By the same argument as \cite[Lemma 2.5]{LouderWilton}, we have a
stacking,
\[
\hat{\Lambda}':\mathbb{S}\to\Gamma\times\mathbb{R}.
\]

\begin{lem}
\label{lem:lambda' is not vertex reducible}Every vertex in $\Gamma$
is covered at least twice by $\Lambda':\mathbb{S}\to\Gamma.$ 
\end{lem}

\begin{proof}
Every vertex $v$ of $\Gamma$ is constructed by gluing together \emph{at
least two }vertices in $\gsigma,$ and $v$ is covered by $\Lambda'$
exactly once for each vertex in $\gsigma$ that has been glued in
the construction of $v$. 
\end{proof}
The connected components of $\mathbb{S}$ are then one of the following:
\begin{itemize}
\item a circle
\item a \emph{closed arc}: a connected and simply connected union of vertices
and interiors of edges, with both ends closed (i.e. a piece).
\end{itemize}
If we also define the following:
\begin{itemize}
\item an \emph{open arc}: a connected and simply connected union of vertices
and interiors of edges, with both ends open
\item a \emph{half--open arc}: a connected and simply connected union of
vertices and interiors of edges, with one end open and one end closed, 
\end{itemize}
then the following lemma is immediate.\footnote{We are especially indebted to Noam Ta Shma for pointing out an error
in our original argument for proving the extension of Wise's $w$--cycles
conjecture and for his feedback on the correction presented from this
point onward.}
\begin{lem}
Every connected component of $\Astack$ or $\B$ is either a circle,
an open arc, a half--open arc or a closed arc. 
\end{lem}

\begin{lem}
\label{lem:A doesn't contain a circle or a long piece}Neither $\Astack$
or $\B$ contain a circle, a closed arc reading out a word of length
$\geq l(w)$ or an open/half--open arc reading out a word of length
$>l(w).$ 
\end{lem}

\begin{proof}
Let $C\subseteq\mathbb{S}$ be a circle, a closed arc reading out
a word of length $\geq l(w)$, or an open/half--open arc reading
out a word of length $>l(w),$ which is contained inside $\Astack$.
Then the map 
\[
\delta\big|_{C}:C\to W
\]
 is surjective. 

$W$ must contain a vertex $x$ belonging to $\mathcal{B}_{\hat{\Lambda}}.$
Then, the preimage $\left(\delta\big|_{C}\right)^{-1}(x)$ contains
a vertex $x'$ belonging to $\mathcal{B}_{\hat{\Lambda}'}$, so that
$x'\in\Astack\cap\mathcal{B}_{\hat{\Lambda}'}.$ This implies that
$\Lambda'(x')\in V(\Gamma)$ is covered exactly once by $\Lambda',$
contradicting Lemma \ref{lem:lambda' is not vertex reducible}. To
prove the lemma for $\B,$ we can swap $\Astack$ and $\B$ in the
proof. 
\end{proof}
So the connected components of $\Astack$ are one of the following: 
\begin{itemize}
\item open arcs that read out words of length $\leq l(w)$,
\item half--open arcs that read out words of length $\leq l(w)$ or 
\item closed arcs that read out words of length $<l(w).$ 
\end{itemize}
The main points here are that every connected component has two ends
(that are either open or closed) and that none of the connected components
can be too long. Moreover, we observe that the only closed ends that
can be in $\Astack$ or $\B$ are those that are the closed ends of
pieces in $\mathbb{S}.$ 

Write 
\[
\mathcal{A}_{\mathrm{end}}\overset{\mathrm{def}}{=}\#\big\{\mathrm{ends\ of\ connected\ components\ of}\ \Astack\big\}
\]
 and 
\[
\mathcal{B}_{\mathrm{end}}\overset{\mathrm{def}}{=}\#\big\{\mathrm{ends\ of\ connected\ components\ of}\ \B\big\}.
\]

\begin{lem}
\label{lem: number of ends in >2k}We have $\mathcal{\mathcal{A}_{\mathrm{end}}}\geq2k$
and $\mathcal{B}_{\mathrm{end}}\geq2k.$
\end{lem}

\begin{proof}
Pick any vertex $w\in W$ such that $w\in\mathcal{A}_{\hat{\Lambda}}.$
Then $w$ has at least $k$ preimages $\delta^{-1}(w)$ in $\mathbb{S},$
all of which must be contained in $\Astack.$ If two such preimages
belong to the same connected component $C$ of $\Astack,$ then $C$
must either be a circle, a closed arc reading out a word of length
$\geq l(w)$ or an open/half--open arc reading out a word of length
$>l(w),$ contradicting Lemma \ref{lem:A doesn't contain a circle or a long piece}.
Thus, there must be at least $k$ connected components in $\Astack,$
so there must be at least $2k$ ends. Repeating the argument with
$w\in\mathcal{B}_{\hat{\Lambda}}$ proves the lemma for $\mathcal{B}_{\mathrm{end}}.$ 
\end{proof}
Write 
\[
\mathcal{A}_{\mathrm{closed}}\overset{\mathrm{def}}{=}\#\big\{\mathrm{closed\ ends\ of\ connected\ components\ of}\ \Astack\big\}
\]
 and 
\[
\mathcal{B}_{\mathrm{closed}}\overset{\mathrm{def}}{=}\big\{\mathrm{closed\ ends\ of\ connected\ components\ of}\ \B\big\}.
\]

\begin{lem}
\label{lem:number of ends is -2=00005Cchi + CLOSED}We have 
\[
\mathcal{A}_{\mathrm{end}}=-2\chi+2\mathcal{A}_{\mathrm{closed}}
\]
and 
\[
\mathcal{\mathcal{B}_{\mathrm{end}}}=-2\chi+2\mathcal{B}_{\mathrm{closed}}.
\]
\end{lem}

\begin{proof}
We prove the lemma for $\Astack.$ Each vertex $v\in V(\Gamma)$ is
covered by $\Astack$ once. If $v$ is covered by a connected component
of $\Astack$ passing \emph{through} one of $v$'s preimages in $\mathbb{S},$
then there must be 
\[
\deg(v)-2
\]
 open ends of $\Astack$ that end at the other preimages of $v$ in
$\mathbb{S},$ to cover the other edges of $\Gamma$ that are incident
at $v.$ 

Otherwise, $v$ is covered by $\Astack$ by a closed end of a connected
component of $\Astack$ that \emph{ends} at one of $v$'s preimages.
In this case, we count exactly $1$ closed end of $\Astack$ that
ends at one of $v$'s preimages and exactly 
\[
\deg(v)-1
\]
open ends of $\Astack$ that end at the other preimages of $v$ in
$\mathbb{S},$ to cover the other edges of $\Gamma$ that are incident
at $v$. 

In total, we see that 
\[
\mathcal{A}_{\mathrm{end}}=\sum_{v}(\deg(v)-2)+\sum_{u}\deg(u),
\]
where the first sum is over $v\in V(\Gamma)$ that are covered by
$\Astack$ by a connected component passing through one of their preimages
and the second sum is over $u\in V(\Gamma)$ that are covered by $\Astack$
by a closed end at one of their preimages. This can be rewritten as
\[
\sum_{v\in V(\Gamma)}(\deg(v)-2)+2\sum_{u}=-2\chi(\Gamma)+2\mathcal{A}_{\mathrm{closed}}.
\]
\end{proof}
Combining Lemma \ref{lem: number of ends in >2k} with Lemma \ref{lem:number of ends is -2=00005Cchi + CLOSED}
implies that 
\[
k\leq-\chi(\Gamma)+\mathcal{A}_{\mathrm{closed}}
\]
and also that 
\[
k\leq-\chi(\Gamma)+\mathcal{B}_{\mathrm{closed}}.
\]
 This implies that 
\begin{equation}
2k\leq-2\chi(\Gamma)+\mathcal{A}_{\mathrm{closed}}+\mathcal{B}_{\mathrm{closed}}\leq-2\chi(\Gamma)+2\sum_{i}\mathrm{del\left(\pi_{i}\right),}\label{eq: 2k < -2chi +2del(=00005Cpi)}
\end{equation}
since the total number of closed ends in $\mathbb{S}$ is $2\sum_{i}\mathrm{del}\left(\pi_{i}\right)$
and each closed end in $\mathbb{S}$ can be, at most, a closed end
in $\Astack$ or a closed end in $\B,$ but not both, since that would
imply that its image in $\Gamma$ is covered only once by the immersion
$\Lambda':\mathbb{S}\to\Gamma$, contradicting Lemma \ref{lem:lambda' is not vertex reducible}.
Combining (\ref{eq: 2k < -2chi +2del(=00005Cpi)}) with (\ref{eq: expected character wth gamsig TILDE})
proves Theorem \ref{thm: word map main theorem}.

\subsection{\label{subsec:ratio of polynomials proof}Proof of Theorem \ref{thm: word map polynomial theorem} }

Let $w\in F_{r}=\left\langle x_{1},\dots,x_{r}\right\rangle $ for
$r$ fixed. Let $m\leq r$ be the minimum number of generators $x_{i_{1}},\dots,x_{i_{m}}$
such that $w$ can be written using the alphabet $\{x_{i_{1}},x_{i_{1}}^{-1},\dots,x_{i_{m}},x_{i_{m}}^{-1}\}$.
So, up to relabeling the generators, $w\in F_{m}=\left\langle x_{1},\dots,x_{m}\right\rangle ,$
with $m\leq r.$ 

Beginning with (\ref{eq: expected character with pi, weingarten and N}),
we have
\[
\begin{aligned} & \mathbb{E}_{w}\left[\chi^{\lambda^{+}(n)}\right]\\
= & \frac{1}{d_{\lambda}}\sum_{\sigma_{f},\tau_{f}}^{\star}\sum_{\pi_{1},\dots,\pi_{l(w)}}^{\leq S_{k}}\left(\prod_{i=1}^{l(w)}c(n,k,\lambda,\pi_{i})\right)\left(\prod_{f\in\{x_{1},\dots,x_{m}\}}\mathrm{Wg}_{n,\left(|w|_{f}k\right)}\left(\sigma_{f},\tau_{f}\right)\right)\\
 & \mathcal{N}\left(\sigma_{f},\tau_{f},\pi_{i}\right).
\end{aligned}
\]

For each $f\in\{x_{1},\dots,x_{m}\},$ for each $\sigma_{f},\tau_{f}\in\refinedpart,$
\[
\mathrm{Wg}_{n,\left(|w|_{f}k\right)}\left(\sigma_{f},\tau_{f}\right)=\sum_{\rho_{f}\leq\sigma_{f}\wedge\tau_{f}}\mu(\rho_{f},\sigma_{f})\mu(\rho_{f},\tau_{f})\frac{1}{(n)_{|\rho_{f}|}}.
\]
 This is equal to 
\[
\begin{aligned} & \frac{C_{1}}{(n)_{|w|_{f}k}}+\frac{C_{2}}{(n)_{|w_{f}|k-1}}+\dots+\frac{C_{|w|_{f}k-|\sigma_{f}\wedge\tau_{f}|+1}}{(n)_{|\sigma_{f}\wedge\tau_{f}|}}\\
= & \frac{g_{\sigma_{f},\tau_{f}}(n)}{(n)_{|w|_{f}k}},
\end{aligned}
\]
where $g_{\sigma_{f},\tau_{f}}(n)$ is a polynomial in $n$ of maximum
degree $|w|_{f}k-|\sigma_{f}\wedge\tau_{f}|$. 

For each $\pi\in\skpart$, we have 
\[
c(n,k,\lambda,\pi)=\frac{d_{\lambda^{+}(n)}(-1)^{|\pi|+k}}{(n)_{|\pi|}}\sum_{\overset{\tau\in S_{k}}{{\scriptstyle {\scriptscriptstyle \iota(\tau)\geq\pi}}}}\chi^{\lambda}(\tau).
\]
Using the hook--length formula, this is equal to 
\[
\frac{\left[\frac{(-1)^{k+|\pi|}}{k!}\sum_{\overset{\tau\in S_{k}}{{\scriptstyle {\scriptscriptstyle \iota(\tau)\geq\pi}}}}\chi^{\lambda}(\tau)\right](n-|\pi|)(n-|\pi|-1)\dots(n-2k+1)}{(n)_{\lambda}},
\]
where $(n)_{\lambda}=\prod_{j=1}^{k}\left(n-k+1-\check{\lambda}_{j}-j\right).$
It follows that 
\[
\prod_{i=1}^{l(w)}c(n,k,\lambda,\pi_{i})=\frac{\prod_{i=1}^{l(w)}h_{k,\lambda,\pi_{i}}(n)}{(n)_{\lambda}^{l(w)}},
\]
where 
\[
h_{k,\lambda,\pi_{i}}(n)=\left[\frac{(-1)^{k+|\pi|}}{k!}\sum_{\overset{\tau\in S_{k}}{{\scriptstyle {\scriptscriptstyle \iota(\tau)\geq\pi}}}}\chi^{\lambda}(\tau)\right](n-|\pi|)(n-|\pi|-1)\dots(n-2k+1)
\]
is a polynomial in $n$ of degree $2k-|\pi_{i}|=k-\mathrm{del}\left(\pi_{i}\right).$
Hence,
\begin{equation}
\mathbb{E}_{w}\left[\chi^{\lambda^{+}(n)}\right]=\frac{p_{k,\lambda,w}\left(n\right)}{\left(n\right)_{\lambda}^{l(w)}\prod_{f}(n)_{|w|_{f}k}},\label{eq: expected character as ratio of polynomials in n}
\end{equation}
where 
\[
p_{k,\lambda,w}\left(n\right)=\frac{1}{d_{\lambda}}\sum_{\sigma_{f}}^{\star}\sum_{\pi_{1},\dots,\pi_{l(w)}}^{\leq S_{k}}\left(\prod_{f}g_{\sigma_{f},\tau_{f}}(n)\right)\left(\prod_{i=1}^{l(w)}h_{k,\lambda,\pi_{i}}(n)\right)\mathcal{N}\left(\sigma_{f},\tau_{f},\pi_{i}\right).
\]
For each collection of $\sigma_{f},\tau_{f},\pi_{i}$:
\begin{itemize}
\item $\left(\prod_{f}g_{f}(n)\right)$ has maximum degree 
\[
\sum_{f}\left(|w|_{f}k-|\sigma_{f}\wedge\tau_{f}|\right)=kl(w)-\sum_{f}|\sigma_{f}\wedge\tau_{f}|,
\]
\item $\left(\prod_{i=1}^{l(w)}h_{k,\lambda,\pi_{i}}(n)\right)$ has maximum
degree 
\[
\sum_{i=1}^{l(w)}\left(k-\mathrm{del}(\pi_{i})\right)=kl(w)-\sum_{i}\mathrm{del(\pi_{i})},
\]
\item $\mathcal{N}\left(\sigma_{f},\tau_{f},\pi_{i}\right)$ has maximum
degree\footnote{This follows from the simple observation that $\gamsig$ has Euler
characteristic $\le\sum_{i}\mathrm{del}(\pi_{i})$, regardless of
if $w$ is a proper power or not.} 
\[
\sum_{i}\mathrm{del}(\pi_{i})+\sum_{f}|\sigma_{f}\wedge\tau_{f}|,
\]
 
\end{itemize}
so that the degree of $p_{k,\lambda,w}(n)$ is less than or equal
to 
\[
2kl(w).
\]
We can rewrite the denominator of (\ref{eq: expected character as ratio of polynomials in n})
using reciprocal polynomials. We have 
\[
\begin{aligned} & \prod_{f}(n)_{|w|_{f}k}\\
= & \prod_{d=1}^{m}\prod_{c=1}^{|w|_{x_{d}}k-1}(n-c)\\
= & n^{kl(w)}\prod_{d=1}^{m}\prod_{c=1}^{|w|_{x_{d}}k-1}\left(1-c\frac{1}{n}\right).
\end{aligned}
\]
Similarly, 
\[
\left(n\right)_{\lambda}^{l(w)}=n^{kl(w)}\left[\prod_{j=1}^{k}\left(1+\left(1+\check{\lambda}_{j}-j-k\right)\frac{1}{n}\right)\right]^{l(w)}.
\]
It follows that 
\[
\mathbb{E}_{w}\left[\chi^{\lambda^{+}(n)}\right]=\frac{\frac{1}{n^{2kl(w)}}p_{k,\lambda,w}\left(n\right)}{\tilde{g}\left(\frac{1}{n}\right)}=\frac{\hat{P}_{w,k,\lambda}\left(\frac{1}{n}\right)}{\tilde{g}\left(\frac{1}{n}\right)},
\]
where 
\[
\tilde{g}\left(x\right)=\prod_{d=1}^{m}\prod_{c=0}^{|w|_{x_{d}}k-1}\left(1-cx\right)\left[\prod_{j=1}^{k}\left(1+\left(1+\check{\lambda}_{j}-j-k\right)x\right)\right]^{l(w)}.
\]
 The numerator $\hat{P}_{w,k,\lambda}\left(\frac{1}{n}\right)$ is
clearly a polynomial in $\frac{1}{n}$ of maximum degree $2kl(w)$. 

Now assume that $l(w)\leq q$, which also implies that $m\leq q.$
Then $\tilde{g}(x)$ always divides 
\[
\hat{g}(x)=\prod_{c=1}^{kq}\left(1-cx\right)^{q}\left[\prod_{j=1}^{k}\left(1+\left(1+\check{\lambda}_{j}-j-k\right)x\right)\right]^{q}.
\]
The sequence $\check{\lambda}_{1}-1,\dots,\check{\lambda}_{k}-k$
is strictly decreasing. Moreover, for any $\lambda\vdash k$ and any
$j\in[k],$ 
\[
1+\check{\lambda}_{j}-j-k\in[1-2k,0],
\]
so that 
\[
\hat{g}(x)\Biggr|\prod_{c=1}^{kq}\left(1-cx\right)^{q}\left[\prod_{j=1}^{2k}\left(1+\left(1-j\right)x\right)\right]^{q}=g_{q,k}(x).
\]

It follows that $\mathbb{E}_{w}\left[\chi^{\lambda^{+}(n)}\right]$
can be written
\[
\frac{P_{w,k,\lambda}\left(\frac{1}{n}\right)}{g_{q,k}\left(\frac{1}{n}\right)},
\]
where 
\[
\begin{aligned} & P_{w,k,\lambda}\left(\frac{1}{n}\right)\\
= & \hat{P}_{w,k,\lambda}\left(\frac{1}{n}\right)\frac{g_{q,k}\left(\frac{1}{n}\right)}{\tilde{g}\left(\frac{1}{n}\right)}
\end{aligned}
.
\]
This is a polynomial in $\frac{1}{n}$ of maximum degree 
\[
\begin{aligned} & 2kl(w)+kq^{2}+2kq-kl(w)-kl(w)+m\\
\leq & 3kq+kq^{2}.
\end{aligned}
\]

\bibliography{C:/Users/ewanc/OneDrive/Documents/Arxiv/library}

\noindent Ewan Cassidy, Department of Mathematical Sciences, Durham
University, Lower Mountjoy, DH1 3LE, Durham, United Kingdom

\noindent ewan.g.cassidy@durham.ac.uk

\noindent (Current institution: Department of Pure Mathematics and
Mathematical Statistics, University of Cambridge, Wilberforce Road,
Cambridge, CB3 0WB

\noindent email: egc45@cam.ac.uk)
\end{document}